\documentclass{article}
\usepackage{hyperref}
\usepackage{amssymb}
\usepackage{amsthm}
\usepackage{microtype}
\usepackage{xspace}
\usepackage{csquotes}
\usepackage{mathtools}
\usepackage{tikz-cd}
\usepackage{bbm}
\usepackage[numbers,sort]{natbib}
\usepackage{calc}
\usepackage{mathpartir}

\newtheorem{thm}{Theorem}

\newtheorem{prop}[thm]{Proposition}

\newtheorem{lemma}[thm]{Lemma}
\newtheorem{theorem}[thm]{Theorem}
\newtheorem{corollary}[thm]{Corollary}
\newtheorem{proposition}[thm]{Proposition}
\theoremstyle{definition}
\newtheorem{defi}[thm]{Definition}
\newtheorem{definition}[thm]{Definition}
\theoremstyle{remark}
\newtheorem{example}[thm]{Example}
\newtheorem{remark}[thm]{Remark}

\title{Invertible cells in $\omega$-categories}
\author{
  Thibaut Benjamin\footnote{University of Cambridge,
    \texttt{tjb201@cam.ac.uk}}
  \and
  Ioannis Markakis\footnote{University of Cambridge,
    \texttt{ioannis.markakis@cl.cam.ac.uk}}
}

\usepackage{macros}

\begin{document}

\maketitle

\begin{abstract}
  We study coinductive invertibility of cells in weak \(\omega\)\-categories.
  We use the inductive presentation of weak \(\omega\)\-categories via an
  adjunction with the category of computads, and show that invertible cells are
  closed under all operations of $\omega$\-categories. Moreover, we give a simple
  criterion for invertibility in computads, together with an algorithm computing
  the data witnessing the invertibility, including the inverse, and the
  cancellation data.
\end{abstract}


\section{Introduction}

Higher category theory is an emergent field with several newfound applications
in computer science and mathematics. In particular, globular higher groupoids
have been used to describe the structure of identity types in Homotopy Type
Theory~\cite{lumsdaine_weak_2009, altenkirch_syntactical_2012,
  vandenberg_types_2011}, establishing a connection with topology, via the
Homotopy Hypothesis. The latter, due to
Grothendieck~\cite{grothendieck_pursuing_1983}, states that weak
\(\omega\)\-groupoids are equivalent to topological spaces up to weak homotopy
equivalence, and is an active topic of research with recent
progress~\cite{henry_homotopy_2023}. Beyond Homotopy Type theory, globular
higher categories have been investigated in connection with higher dimensional
rewriting~\cite{mimram_3dimensional_2014} and topological quantum field
theory~\cite{bartlett_modular_2015}, and in
homology~\cite{lafont_polygraphic_2009}.

Higher categories can be described starting from different shapes, and using
several variations on their axioms. They can have various level of strictness,
and can be truncated. Surveys on different definitions of higher categories have
been published by Cheng and Lauda~\cite{cheng_higherdimensional_2004}, and
Leinster~\cite{leinster_survey_2001}. In this article, we focus on globular weak
\((\infty,\infty)\)\-categories, henceforth called \(\omega\)\-categories. Those
were originally introduced by Batanin~\cite{batanin_monoidal_1998}, and then
Leinster~\cite{leinster_higher_2004} as algebras for the monad \(T\) defined as
the initial contractible operad. Then,
Matsiniotis~\cite{maltsiniotis_grothendieck_2010} proposed another definition in
terms of models of a globular theory, adapting an idea due to
Grothendieck~\cite{grothendieck_pursuing_1983} for \(\omega\)\-groupoids. Those
two definition have been proven equivalent by Ara~\cite{ara_infty_2010} and
Bourke~\cite{bourke_iterated_2020}. More recently, Finster and
Mimram~\cite{finster_typetheoretical_2017} have proposed a more syntactic
definition, viewing \(\omega\)\-categories as the models of a dependent type
theory called \catt{}. Benjamin, Finster and Mimram~\cite{benjamin_globular_2021}
have proven this definition to be equivalent to that of Grothendieck and
Maltsiniotis. Inspired by this type theory, Dean et
al.~\cite{dean_computads_2024} have proposed an equivalent description of the
monad \(T\) of Batanin and Leinster, in terms of an adjunction
\[
  \free : \glob \rightleftarrows \comp : \cell
\]
between globular sets and a category of computads. A direct comparison between
the type theory \catt and the computads
introduced by Dean et al. has been established by Benjamin, Markakis and
Sarti~\cite{benjamin_catt_2024}.

Understanding the homotopy theory of \(\omega\)\-categories is one of the major
open problems of higher category. The homotopy theory of strict
\(\omega\)\-categories has been studied by Lafont et
al.~\cite{lafont_folk_2010}, motivated by its strong connection with homology
and its application to rewriting theory, continuing the work of
Squier~\cite{squier_word_1987}. In particular, they have defined a model
structure in which computads are the cofibrant objects, and for which the
description of the weak equivalences relies heavily on the notion of weakly
invertible cells. It is conjectured that a similar model structure should exist
for weak \(\omega\)\-groupoids and weak \(\omega\)\-categories. It was proven by
Henry~\cite{henry_algebraic_2016} that constructing this model structure on \(\omega\)\-groupoids is the missing
piece to prove the homotopy hypothesis. Partial
results in this direction have been worked out, mainly by Henry and
Lanary~\cite{henry_homotopy_2023}, who have proven the homotopy hypothesis in
dimension \(3\). On weak \(\omega\)\-categories, a weak factorization
corresponding to the putative cofibrations and trivial fibrations is known. Dean
et al.~\cite{dean_computads_2024} and Markakis~\cite{markakis_computads_2024}
have proven that computads are exactly the cofibrant
objects for this conjectured model structure.

In this article, we investigate weakly invertible cells in
\(\omega\)\-categories, a notion used in the characterisation of the weak
equivalences for the conjectured model structure. Such cells have been defined
coinductively for a broad class of globular higher structures by
Cheng~\cite{cheng_omega_2007}, and studied in the case of
strict \(\omega\)\-categories by Lafont et al.~\cite{lafont_folk_2010}.
More recently, Rice~\cite{rice_coinductive_2020} has compared this coinductive
notion of invertibility to other proposed notions of invertibility.

Fujii, Hoshino and
Maehara~\cite{fujii_weakly_2023} have also studied coinductively invertible
cells in \(\omega\)\-categories, and shown that they are closed under all
operations. Nonetheless, our work differs significantly from theirs in various
aspects. We use the inductive presentation of Leinster's monad by Dean et
al.~\cite{dean_computads_2024},
which provides us with an explicit syntax to work with cells of
\(\omega\)\-categories. This allows us to give a more precise version of their
main theorem, together with an elementary proof of it. Our work further provides
a syntactic criterion deciding the invertibility of a cell in a finite
dimensional computad, as well as a
structurally recursive algorithm computing the inverse of a cell and the
cancellation data attached to it.

The complexity of out work originates from that of
\(\omega\)\-categories. To tackle this complexity, we expand the syntax provided
by the inductive presentation of the free \(\omega\)\-category monad with
reusable meta-operations that produce new operations from existing ones,
continuing our previous work~\cite{benjamin_opposites_2024}. More precisely, we
will use the \emph{opposite} and \emph{suspension} operations introduced there,
which amount to interpreting an operation as an operation of the opposite or hom
\(\omega\)\-category respectively. We will also introduce two new operations, the
\emph{functorialisation}, which has only been studied using the type theory
\catt~\cite{benjamin_type_2020}, and the \emph{chain reduction} operation. The
former relies on the idea that each operation of \(\omega\)\-categories is
functorial, and hence can be applied to higher dimensional cells as well. The
latter, replaces a composition operation with an equivalent more biased one over
a simpler pasting diagram. Finally, to tackle the combinatorial complexity of
computing the invertibility data for a composite of invertible cells, we further
need to introduce some ad-hoc operations, specific to the problem. Those
operations allow us to cancel the composition of a sequence of cells with their
inverses. Due to the shape of the computad they live over, we call those
operations \emph{telescopes}.

Using those operations, we give an elementary proof that any composite of
invertible cells in an \(\omega\)\-category is again invertible. We believe that
the simplicity of this proof is strong evidence that the syntactic approach to
\(\omega\)\-categories is promising, and can help further develop the theory of
\(\omega\)\-categories. Proofs built using this syntactic presentation can often
lead to algorithms, or meta-operations expanding the language itself. For
example, given a cell in a computad satisfying our invertibility criterion, our
proof gives a recipe to construct its inverse and the invertible cancellation
witnesses. This procedure has been implemented as an extension of the proof
assistant \catt{}\footnote{\url{http://www.github.com/thibautbenjamin/catt}},
based on the dependent type theory with the same name, dedicated to working in
the language of \(\omega\)\-categories. With this new feature, a user can input
a term corresponding to an invertible cell, and the proof assistant
automatically computes the term corresponding to the chosen inverse, or the term
corresponding to any of its invertibility data, more generally.

\subsection*{Overview of the paper}
In Section~\ref{sec:trees}, we recall the notion of globular pasting diagram and
define operations on them making them into a free strict \(\omega\)\-category.
Section~\ref{sec:weak-cats} then recalls the definition of computads and the
free $\omega$\-category monad on globular sets given by Dean et
al.~\cite{dean_computads_2024}. Section~\ref{sec:constructions} is dedicated to
defining several constructions in \(\omega\)\-categories that allow us to expand
the language in which we work. In Section~\ref{sec:inverse-composite}, we show
our main theorem stating that a composite of invertible cells is invertible.
Finally, Section~\ref{sec:implem} presents and evaluates the implementation of
our main result in the proof assistant \catt{}.

\subsection*{Acknowledgements}

The authors would like to thank Prof. Jamie Vicary for his support during this
project. Furthermore, the second author would like to acknowledge funding from
the Onassis foundation - Scholarship ID: F ZQ 039-1/2020-2021.



\section{Globular pasting diagrams}\label{sec:trees}

In this background section, we will recall globular sets and globular pasting
diagrams, the underlying shapes and the arities of the operations of
\(\omega\)\-categories respectively. Globular pasting diagrams are a family of
globular sets that has been studied extensively under different presentations,
namely globular cardinals~\cite{street_petit_2000},
globular sums~\cite{ara_infty_2010}, or pasting
diagrams~\cite{leinster_higher_2004}. For a survey of those presentations and
their equivalence, we refer to the work of
Weber~\cite[Section~4]{weber_generic_2004}. Here we will expand on the
presentation of Dean et al.~\cite{dean_computads_2024} by also describing the
composition operations of trees in their setting.

\subsection{Globular sets}\label{subsec:gset}

The category \(\G\) of globes has objects the natural numbers, and morphisms
generated by the \emph{cosource} and \emph{cotarget}
\(s,t \colon n\to (n+1)\) under the \emph{coglobularity} relations:
\begin{align*}
  s\circ s &= t \circ s & s\circ t &= t \circ t.
\end{align*}
The category \(\glob\) of globular sets is the category of presheaves on \(\G\).
More explicitly, a globular set \(X \colon \G^{\op}\to \Set\) consists of a
set \(X_n\) for every \(n\in \N\) together with \emph{source} and \emph{target}
functions \(\src,\tgt \colon X_{n}\to X_{n+1}\) satisfying the
\emph{globularity} relations:
\begin{align*}
  \src\circ\src &= \src\circ \tgt & \tgt\circ\src &=\tgt\circ\tgt.
\end{align*}
We will call the elements of \(X_n\) the \(n\)\-cells of \(X\). We also define
the \emph{\(m\)\-source} and \emph{\(m\)\-target} of an \(n\)\-cell \(x\in X_n\)
for \(m < n\) by iterating the source and target functions:
\begin{align*}
  \src_k x =  &= \src(\cdots(\src x)) &
  \tgt_k x =  &= \tgt(\cdots(\tgt x)).
\end{align*}
We will say that a pair of \(n\)\-cells are \emph{parallel} when they have the
same source and target, where, by convention, all \(0\)\-cells are parallel.

The \emph{\(n\)\-disk} \(\disk n\) for \(n\in\N\) is the representable globular
set \(\G(-,n)\). The \emph{\(n\)\-sphere} \(\sphere n\) for \(n \ge -1\) is
defined recursively with an inclusion \(\iota_n\colon\sphere{n-1} \to \disk{n}\)
via the following pushout diagram
\[\begin{tikzcd}
	\sphere{n-1}
  \dar[swap]{\iota_n}
	\rar{\iota_n} &
  \disk n
  \dar[dashed]{}
  \ar[ddr, bend left, "s"] \\
  \disk n
  \rar[dashed]{}
  \ar[drr, bend right, "t"] &
  \sphere n
  \drar[dashed]{\iota_{n+1}}
  \pomark \\
  && \disk{n+1}
\end{tikzcd}\]
starting from \(\sphere {-1} = \emptyset\) being the initial globular set. By
the Yoneda lemma, morphisms \(\disk n\to X\) are in natural bijection to
\(n\)\-cells of \(X\), while morphisms \(\sphere n\to X\) are in natural
bijection to pairs of parallel \(n\)\-cells of \(X\). Under those bijections,
composition with the inclusion \(\iota_n\) sends a cell to its source and
target, which are parallel by the globularity relations.

Globular pasting diagrams are a family of globular sets defined recursively,
using the suspension and the wedge sum of globular sets. The \emph{suspension}
of a globular set \(X\) is the globular set \(\susp X\) with cells given by
\begin{align*}
   (\susp X)_0 &= \{v_-,v_+\} & (\susp X)_{n+1} &= X_n.
\end{align*}
Its source and target functions are given by those of \(X\), with \(v_-\) being
the source and \(v_+\) being the target of every \(1\)\-cell. The
\emph{wedge sum} \(X\vee Y\) of a pair of globular sets \(X\) and \(Y\) with
respect to chosen \(0\)\-cells \(x_-,x_+\in X_0\) and \(y_-,y_+\in Y_0\) is
defined to be the following pushout in \(\glob\):
\[
  \begin{tikzcd}[column sep = small, row sep = small]
    & & X\vee Y & & \\
    & X
    \ar["\inc_1", ru, dashed]
    & &
    Y\ar["\inc_2"', lu, dashed] & \\
    \disk{0}
    \ar[ru,"x_-"] & &
    \disk{0}
    \ar[ru,"y_-"]
    \ar[lu,"x_+"']
    \ar[uu,phantom,"\llcorner"{rotate = 45, very near end}] & &
    \disk{0}
    \ar[lu,"y_+"']
  \end{tikzcd}
\]
The wedge sum defines a monoidal product in the category of globular sets with
two chosen \(0\)\-cells with unit the \(0\)\-disk, being the composition of
cospans of globular sets.

\subsection{Batanin trees and their positions}\label{subsec:trees-pos}

Batanin was the first to observe that globular pasting diagrams are indexed by
isomorphisms classes of rooted planar trees and to give a combinatorial
description of them~\cite{batanin_monoidal_1998}. As explained by
Leinster~\cite{leinster_higher_2004}, there exists one such tree of
dimension $0$, while trees of dimension at most $n+1$ are precisely list of
trees of dimension at most $n$. This leads to the following \emph{inductive}
definition of rooted planar trees, which we will call \emph{Batanin trees}.

\begin{defi}\label{def:tree}
  The set of Batanin trees is inductively defined by one rule: there exists a
  Batanin tree \(\btlist L\) for every list \(L\) of Batanin trees.
\end{defi}

\noindent The rule specifies that the set \(\bat\) of Batanin trees is equipped
with a function \(\btlist \colon \List(\bat)\to \bat\) where \(\List\) is the
free monoid endofunctor
\[
  \List X = \coprod_{n\in \N} X^n.
\]
Being inductively generated by this rule means precisely that the pair
\((\bat,\btlist)\) is the initial algebra for the endofunctor \(\List\). In
particular, there exists a tree $\bt{}$ corresponding to the empty list, and
using this tree, we can define more complicated trees, such as the tree
\[B=\bt{\bt{\bt{},\bt{}},\bt{}}.\]
We visualise those trees by letting \(\btlist L\) be the tree with a new root
and with branches given by \(L\). For example, \(\bt{}\) is the tree with one
vertex and no branches, while the tree \(B\) above can be visualised as follows:
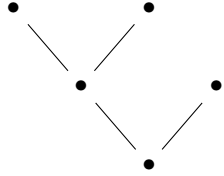
\[
  \begin{tikzcd}[column sep = small]
    \bullet & & \bullet\\
    & \bullet \ar[ul,-]\ar[ur,-] & & \bullet \\
    & & \bullet \ar[ul,-] \ar[ur,-]
  \end{tikzcd}
\]
The \emph{dimension} of a Batanin tree is the height of the corresponding planar
tree, or equivalently the maximum of the dimension of its positions,
defined below. It can be computed recursively by
\[
  \dim (\bt{B_1, \ldots, B_n}) = \max(\dim B_1 +1, \ldots, \dim B_n +1).
\]
In particular, \(\bt{}\) is the unique Batanin tree of dimension \(0\).

\begin{defi}\label{def:pos}
  The globular set of \emph{positions} of a Batanin tree $B$ is the globular
  set \(\Pos{B}\) defined inductively by the formula
  \[
    \Pos{\bt{B_{1}\ldots,B_{n}}} =
    \bigvee_{i=1}^{n} \susp \Pos{B_{i}}.
  \]
\end{defi}

\noindent
This is the globular pasting diagram corresponding to the planar tree \(B\) as
explained by Leinster~\cite[Appendix~F.2]{leinster_higher_2004}. The positions
of a Batanin tree correspond to sectors of the corresponding planar
tree~\cite{berger_cellular_2002}. For example, the globular set of positions of
the tree \(B\) above is the following one
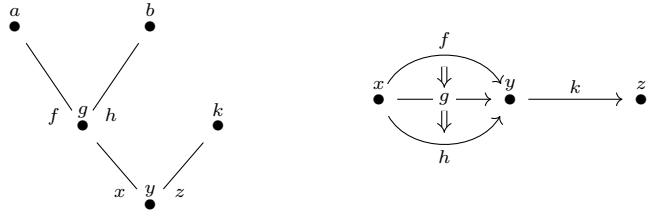
\begin{align*}
  \begin{tikzcd}[ampersand replacement = \&, column sep = small]
    \overset{a}{\bullet} \&\& \overset{b}{\bullet} \\
    \&\vphantom{\bullet}\smash{\overset{g}{\bullet}}\&\& \overset{k}{\bullet}\\
    \&\&\vphantom{\bullet}\smash{\overset{y}{\bullet}}
    \arrow[from = 2-2, to = 1-1, no head, "f" very near start]
    \arrow[from = 2-2, to = 1-3, no head, swap, "h" very near start]
    \arrow[from = 3-3, to = 2-2, no head, "x" very near start]
    \arrow[from = 3-3, to = 2-4, no head, swap, "z" very near start]
  \end{tikzcd} &&
  \begin{tikzcd}[ampersand replacement = \&, column sep = large]
    \vphantom{\bullet}\smash{\overset{x}{\bullet}} \&
    \vphantom{\bullet}\smash{\overset{y}{\bullet}} \&
    \vphantom{\bullet}\smash{\overset{z}{\bullet}}
    \arrow[from=1-1, to=1-2, bend left=60, "f", ""{name = f, anchor = center}]
    \arrow[from=1-1, to=1-2, "g"{description}, ""{name = g, anchor = center}]
    \arrow[from=1-1, to=1-2, bend right=60, "h"', ""{name = h, anchor = center}]
    \arrow[from=f, to=g, Rightarrow, shorten <= 2pt, shorten >= 2pt]
    \arrow[from=g, to=h, Rightarrow, shorten <= 2pt, shorten >= 2pt]
    \arrow[from=1-2, to=1-3, "k"]
  \end{tikzcd}
\end{align*}
Here the positions $f,g,h,a,b$ are the positions of the left branch
$\bt{\bt{},\bt{}}$, while $k$ is the position of the right branch $\bt{}$. The
$0$\-positions $x,y,z$ are the new cells created by the suspension
operation.

\begin{defi}\label{def:lmax-pos}
  A position $p\in \Pos[k]{B}$ of a Batanin tree $B$ will be called
  \emph{locally maximal} when it is not the source, nor the target of
  another position. It will be called \emph{maximal} when \(k = \dim B\).
\end{defi}

\noindent
One can show by induction that the set \(\Pos[k]{B}\) of \(k\)\-positions of a
Batanin tree \(B\) is empty if and only if \(k > \dim B\), so maximal positions
are locally maximal. Moreover, the locally maximal positions of a tree can be
defined recursively by letting the unique position of \(\bt{}\) be locally
maximal, and letting the position \(\inc_j(p)\) of \(\bt{B_1,\ldots,B_n}\) be
locally maximal when \(p\) is locally maximal in \(B_j\), where \(\inc_j\) is
the inclusion of the \(j\)\-th summand in the wedge sum
\[\inc_j : \Sigma\Pos{B_j} \to \bigvee_{i=1}^n \Sigma \Pos{B_i}.\]

\begin{example}\label{ex:susp-tree}
  The \emph{suspension} of a Batanin tree \(B\) is the tree
  \(\susp B = \bt{B}\). By construction, \(\susp \Pos{B} = \Pos{\susp B}\), so
  globular pasting diagrams are preserved by the suspension operation.
\end{example}

\begin{example}\label{ex:disk-tree}
  Representable globular sets are globular pasting diagrams. More specifically,
  we can define recursively on \(n\in \N\) a tree \(D_n\) together with an
  isomorphism \(\Pos{D_n} \cong \disk n\). We start by letting \(D_0 = \bt{}\)
  and the isomorphism being the identity of \(\disk 0\), and then proceed
  to define \(D_{n+1} = \susp D_n\) and the isomorphism to be the composite
  \[
    \Pos{D_{n+1}} = \susp \Pos{D_n} \cong \susp \disk n \cong \disk{n+1}
  \]
  where the last isomorphism sends the unique top-dimensional cell of
  \(\disk n\) to the unique top-dimensional cell of \(\disk{n+1}\).
\end{example}

\subsection{Operations on pasting diagrams}\label{subsec:operations-trees}

Globular pasting diagrams familially represent the free strict
\(\omega\)\-category monad. In this section, we will describe the free strict
\(\omega\)\-category on a globular set, and discuss briefly the monad
multiplication.

\begin{defi}\label{def:bdry-tree}
  The \emph{\(k\)\-boundary} of a Batanin tree \(B\) for \(k \in \N\) is the
  Batanin tree defined recursively by the following formulae
  \begin{align*}
    \bdry[0]{\bt{B_1,\dots,B_n}} &= \bt{} \\
    \bdry[k+1]{\bt{B_1,\dots,B_n}} &= \bt{\bdry[k]{B_1}, \dots,\bdry[k]{B_n}}
  \end{align*}
\end{defi}

\noindent
On the level of rooted, planar trees, one can show inductively that the tree
\(\bdry[k]{B}\) is obtained from \(B\) by removing all nodes of distance at
least \(k\) from the root. In terms of pasting diagrams, it is obtained by
removing all positions of dimension above \(k\) and identifying parallel
\(k\)\-positions. It follows that the positions of the \(k\)\-boundary can be
included back into the positions of the original tree in two ways by picking
the leftmost and rightmost position for every branch. More formally, the
\emph{\(k\)\-cosource} and \emph{\(k\)\-cotarget}
\[
  \srcps[k]{B}, \tgtps[k]{B} : \Pos{\bdry[k]{B}}\to \Pos{B}
\]
are defined recursively for a tree \(B = \bt{B_1,\dots,B_n}\) by the following
formulae:
\begin{align*}
  \srcps[0]{B} &=  \inc_1(v_-) &
  \tgtps[0]{B} &=  \inc_n(v_+) \\
  \srcps[k+1]{B} &= \bigvee_{i=1}^n \susp \srcps[k]{B_i} &
  \tgtps[k+1]{B} &= \bigvee_{i=1}^n \susp \tgtps[k]{B_i}.
\end{align*}
As the name suggests, the \(k\)\-cosource and \(k\)\-cotarget satisfy the
coglobularity relations. Considering our example Batanin tree
$B$, we have that
\begin{align*}
  \bdry[1]{B} =
  \begin{tikzcd}[ampersand replacement = \&, column sep = small]
    \& \bullet \& \& \bullet \\
    \& \& \bullet \ar[ul,-] \ar[ur,-]
  \end{tikzcd}&&
  \Pos{\bdry[1]{B}} =
  \begin{tikzcd}[ampersand replacement = \&]
    \vphantom{\bullet}\smash{\overset{x}{\bullet}} \ar[r, "f"] \&
    \vphantom{\bullet}\smash{\overset{y}{\bullet}}\ar[r, "k"] \&
    \vphantom{\bullet}\smash{\overset{z}{\bullet}}
  \end{tikzcd}
\end{align*}
The source inclusion \(\srcps[1]{B}\) is the one given by the names of the
positions, while the target inclusion \(\tgtps[1]{B}\)
is the one sending \(f\) to \(h\) and the other positions to the ones with the
same name.

\begin{remark}
  To simplify the notation, we will denote by \(\partial B\) the boundary
  \(\partial_{\dim B - 1} B\), and we will denote the corresponding source and
  target inclusions by \(\srcps{B}\) and \(\tgtps{B}\) respectively.
\end{remark}

\begin{definition}\label{def:bin-comp-tree}
  The \emph{\(k\)\-composition} of a pair of Batanin trees \(B\) and \(B'\)
  sharing a common \(k\)\-boundary is the Batanin tree \(\graft{k}{B}{B'}\)
  defined recursively by the formulae
  \begin{align*}
    \graft{0}{\bt{B_{1},\ldots,B_{n}}}{\bt{B'_{1},\ldots,B'_{m}}}
    &= \bt{B_{1},\ldots,B_{n},B'_{1},\ldots,B'_{m}} \\
    \graft{k+1}{\bt{B_{1},\ldots,B_{n}}}{\bt{B'_{1},\ldots,B'_{n}}}
    &= \bt{\graft{k}{B_{1}}{B'_{1}},\ldots,\graft{k}{B_{n}}{B'_{n}}}
  \end{align*}
\end{definition}

\noindent
We observe first that those equations completely determine the composition of
Batanin trees, since for a pair of Batanin trees \(\btlist L\) and
\(\btlist L'\) to share a common positive-dimensional boundary, the lists \(L\)
and \(L'\) must have the same length. The composition of a pair of Batanin trees
along a common $k$\-boundary amounts to appending the branches of the two trees
at every node of distance $k$ from the root. On the level of pasting diagrams,
it realises the glueing of the corresponding pasting diagrams along their common
boundary, as will be shown in Proposition~\ref{prop:bin-comp-tree-po}. For
example, consider the Batanin trees
\begin{align*}
  B &= \bt{\bt{},\bt{\bt{{\color{red}\bt{}}}, \bt{{\color{red}\bt{},\bt{}}}}} \\
  B' &= \bt{\bt{},\bt{\bt{{\color{blue}\bt{},\bt{\bt{}}}}, \bt{}}}
\end{align*}
They share a common \(2\)\-boundary, which is the tree
\[
  \partial_2B = \partial_2B' = \bt{\bt{},\bt{\bt{}, \bt{}}}
\]
and their \(2\)\-composition is the following tree
\[
  \graft{2}{B}{B'} =
    \bt{\bt{},\bt{\bt{{\color{red}\bt{}},{\color{blue}\bt{},\bt{\bt{}}}},
    \bt{{\color{red}\bt{},\bt{}}}}}.
\]
This process may be visualised as follows:
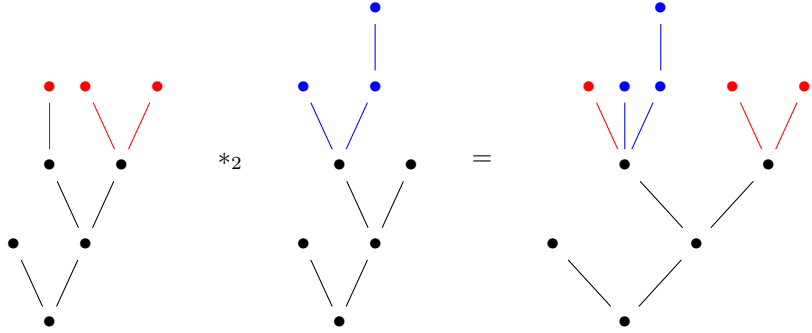
\begin{align*}
  \begin{tikzcd}[column sep = 0em, ampersand replacement = \&]
    \vphantom{\bullet} \\
    \& {\color{red} \bullet} \& {\color{red} \bullet}
    \& \& {\color{red} \bullet} \\
    \& \bullet\ar[u,-, red] \&  \& \bullet \ar[ul,-, red]\ar[ur,-, red] \\
    \bullet \& \& \bullet \ar[ul,-]\ar[ur,-] \& \\
    \& \bullet \ar[ul,-]\ar[ur,-] \& \&
  \end{tikzcd}&&\graft{2}{}{}&&
  \begin{tikzcd}[column sep = 0,ampersand replacement = \&]
    \& \& {\color{blue} \bullet} \& \& \\
    {\color{blue} \bullet}
    \& \& {\color{blue} \bullet}\ar[u,-,blue] \& \& \\
    \& \bullet\ar[ul,-, blue]\ar[ur,-, blue] \&  \& \bullet \\
    \bullet \& \& \bullet \ar[ul,-]\ar[ur,-] \& \\
                            \& \bullet \ar[ul,-]\ar[ur,-] \& \&
  \end{tikzcd}&&=&&
  \begin{tikzcd}[column sep = 0,ampersand replacement = \&]
    \&\&\& {\color{blue} \bullet}\&\&\&\&\\
    \& {\color{red} \bullet}
    \& {\color{blue} \bullet}
    \& {\color{blue} \bullet}\ar[u,-,blue]
    \&\& {\color{red} \bullet}
    \& \& {\color{red} \bullet} \\
    \&\& \bullet\ar[ul,-, red]\ar[u,-, blue]\ar[ur,-, blue] \&\&
    \&\& \bullet\ar[ul,-,red]\ar[ur,-,red] \\
    \bullet \&\&\&\& \bullet\ar[urr,-]\ar[ull,-] \&\& \\
    \&\& \bullet\ar[urr,-]\ar[ull,-]  \&\&\&\&
  \end{tikzcd}
\end{align*}

\begin{proposition}\label{prop:bin-comp-tree-po}
  For every \(k\in \N\) and every pair of Batanin trees \(B\) and \(B'\) sharing
  a common \(k\)\-boundary, there exists a pushout square of the form
  \[\begin{tikzcd}
    \Pos{\partial_k B} & \Pos{\partial_k B'} & \Pos{B'} \\
    \Pos{B} && \Pos{\graft{k}{B}{B'}}
    \arrow["{\tgtps[k]{B}}"', from=1-1, to=2-1]
    \arrow["{\srcps[k]{B'}}", from=1-2, to=1-3]
    \arrow[equals, from=1-1, to=1-2]
    \arrow["\inc_{k,B,B'}^+", from=1-3, to=2-3]
    \arrow["\inc_{k,B,B'}^-", from=2-1, to=2-3]
  \end{tikzcd}\]
\end{proposition}
\begin{proof}
  We will construct this pushout diagrams by induction on the Batanin trees
  \(B = \bt{B_1,\dots,B_n}\) and \(B' = \bt{B_1',\dots,B_m'}\) and on
  \(k\in\N\). We define first \(\inc_{0,B,B'}^-\) be the morphism induced by the
  inclusions \(\inc_1, \dots, \inc_n\) of the components of the wedge sum, and
  \(\inc_{0,B,B'}^+\) be the morphism induced by the inclusions
  \(\inc_{n+1},\dots,\inc_{n+m}\). The resulting square is a pushout by the
  definition and associativity of the wedge sum operation. We then define
  recursively
  \[
    \inc_{k+1,B,B'}^\pm = \bigvee_{i=1}^n \susp \inc_{k,B_i,B_i'}^\pm.
  \]
  The resulting square is a pushout, since both the wedge sum and the suspension
  operations preserve connected colimits; the former because it factors as a
  left adjoint \(\glob\to\glob_{**}\)~\cite[Section~2]{benjamin_opposites_2024}
  followed by the forgetful functor from bipointed globular sets to globular
  sets, and the latter by distributivity of colimits over colimits.
\end{proof}

\begin{definition}\label{def:free-strict-cat}
  The \emph{free strict \(\omega\)\-category} on a globular set \(X\) is the
  strict \(\omega\)\-category \(F^{\str}X\) consists of the globular set
  \[(F^{\str}X)_n = \coprod_{\dim B \le n} \glob(\Pos B , n)\]
  with source and target functions given by
  \begin{align*}
    \src_k(B,f) &= (\partial_k B,f \circ \srcps[k]{B}) &
    \tgt_k(B,f) &= (\partial_k B,f \circ \tgtps[k]{B}),
  \end{align*}
  with identity operations given by the obvious subset inclusions, and with
  composition operations given by
  \[
    \graft{k}{(B,f)}{(B',f')} = (\graft{k}{B}{B'}, \ideal{f,f'}),
  \]
  where \(\ideal{f,f'}\) is the morphism out of the pushout of
  Proposition~\ref{prop:bin-comp-tree-po}.
\end{definition}

\noindent
It is out of the scope of this paper to show that \(F^{\str}X\) is a strict
\(\omega\)\-category.  Nonetheless, we will use freely that the operations
\(\graft{k}{}{}\) are associative and unital, and that they satisfy the
interchange law, since those properties have been shown for example by
Leinster~\cite[Appendix~F.2]{leinster_higher_2004}. An alternative way to prove
those axioms is to restate them as equations between the cosources, cotargets,
and the pushout inclusions, and show them by induction on the trees involved.

As the name suggests, the functor \(F^{\str}\) is left adjoint to the
underlying globular set functor \(U^{\str}\) forgetting the identity and
composition operations of a strict \(\omega\)\-category. We will denote the
induced monad by \(T^{\str}\colon \glob \to \glob\). As an endofunctor,
\(T^{\str}\) is familially represented by the collection of Batanin trees,
so it is cartesian, finitary and coproduct
preserving~\cite[Theorem~F.2.2]{leinster_higher_2004}. The unit of the monad
\(\eta^{\str}\colon \id \to T^{\str}\) is the natural transformation, whose
component at a globular set \(X\), sends an \(n\)\-cell \(x\in X_n\) to morphism
\(\Pos {D_n}\to \disk n \to X\), where the first morphism is the isomorphism of
Example~\ref{ex:disk-tree} and the second is obtained by the Yoneda lemma.
The monad multiplication is described by the following proposition of
Weber~\cite[Proposition~4.7]{weber_generic_2004}.

To state the proposition, we observe that the free strict \(\omega\)\-category
on the terminal globular set \(T^{\str}1\) has as \(n\)\-cells Batanin trees of
dimension at most \(n\), and it has the \(k\)\-boundary as its \(k\)\-source and
\(k\)\-target function, so we will denote it by \(\bat\) when no confusion may
arise. The assignment of the globular set of positions of a Batanin tree can be
seen as a functor
\[
  \PosOp \colon \smallint\bat \to \glob,
\]
from the category of elements of \(\bat\) sending an object \((n, B)\) to the
globular set \(\Pos B\), and the generating morphisms
\(s,t\colon (n,\partial_n B)\to (n+1,B)\) to \(s_n^B\) and \(t_n^B\)
respectively.

\begin{proposition}\label{prop:unb-comp-tree}
  Let \(f \colon \Pos B\to \bat\) a morphism of globular sets. Then there exists
  a Batanin tree \(\mu_1^{\str}(f)\) of dimension at most \(\dim B\) such that
  \[
    \Pos{\mu_1^{\str}(B,f)} = \colim_{(k,p)\in \smallint \Pos{B}} \Pos{f_k(p)}
  \]
  we will denote the canonical cocone by
  \(j^f \colon \PosOp\circ \smallint f \Rightarrow \Pos{\mu_1^{\str}(B,f)}\).
  Moreover, this construction commutes with cosource and cotargets.
\end{proposition}

\noindent
A cell of \(x\in T^{\str}T^{\str}X\) is a pair \((B,f)\) where
\(f\colon \Pos B\to T^{\str}X\). Writing \(f(p) = (f^1(p), f^2(p))\) where
\(f^1(p)\) is a Batanin tree and \(f^2(p) \colon \Pos{f^1(p)} \to X\), we see
that \(f\) is a morphism of globular sets if and only if \(f^1\) is a morphism
and \(f^2\) is a cocone \(\PosOp\circ \smallint f^1 \Rightarrow X\). The monad
multiplication \(\mu^{\str}_X\colon T^{\str}T^{\str}X\to T^{\str}X\) is then
given by
\[
  \mu^{\str}_X(B,f) = (\mu_1^{\str}(B,f^1), \ideal{f^2})
\]
where \(\ideal{f^2}\) is the morphism out of the colimit induced by the cocone
\(f^2\). The last part of Proposition~\ref{prop:unb-comp-tree} ensures that
each \(\mu^{\str}_X\) is a morphism of globular sets.



\section{Weak \texorpdfstring{\(\omega\)}{ω}-categories}
\label{sec:weak-cats}

As seen from Definition~\ref{def:free-strict-cat}, strict
\(\omega\)\-categories are globular sets equipped with a unique way to compose
diagrams of cells indexed by a globular pasting diagram. Weak
\(\omega\)\-categories are a generalisation, in which such diagrams admit a
unique composite up to an invertible higher cell. Following
Leinster~\cite{leinster_higher_2004}, we define weak \(\omega\)\-categories as
algebras for certain monad on globular sets, the initial contractible
globular operad. We will use the description of this monad given by Dean et
al.~\cite{dean_computads_2024}, in terms of an adjunction
\[\free\colon \glob \rightleftarrows \comp \colon \cell\]
with the category
\(\comp\) of computads. In this section, we recall this description of the
\(\omega\)\-category monad, and certain operations on \(\omega\)\-categories
and their computads that were introduced in our previous
work~\cite{benjamin_opposites_2024}.

\subsection{Computads}\label{subsec:computads}

Computads are generating data for \(\omega\)\-categories, and they are defined
by induction on the dimension \(n\in\N\). More
precisely, we define mutually inductively the category \(\comp_n\) of
\(n\)\-computads together with four functors
\begin{align*}
  \free_n &\colon \glob \to \comp_n & \cell_n &\colon \comp_n \to \Set \\
  u_n &\colon \comp_n \to \comp_{n-1} & \type_n &\colon \comp_n \to \Set
\end{align*}
three natural transformations
\begin{align*}
  \ty_n &\colon\cell_n \Rightarrow \type_{n-1}u_n &
  \pr_1, \pr_2 &\colon \type_n \Rightarrow \cell_n
\end{align*}
and an auxiliary subset \(\Full_n(B)\subseteq \type_n\free_n(\Pos B)\) of
spheres for every Batanin tree \(B\) that we will call full. The functor \(u_n\)
forgets the top-dimensional generators of a computad. The functor \(\free_n\)
views a globular set as a computad whose generators are its cells. The functor
\(\cell_n\) returns the \(n\)\-cells of the \(\omega\)\-category generated by
the computad, while the functor \(\type_n\) returns pairs of parallel
\(n\)\-cells. The projections \(\pr_i\) pick the first and second cell of a
parallel pair, while the natural transformation \(\ty_n\) assigns to each cell
its source and target. For the base case, we let \(\comp_{-1}\) be the terminal
category and \(\type_{-1}\) the functor picking the terminal set.

An \(n\)\-computad is a triple \((C_{n-1}, V_n^C, \phi_n^C)\) consisting of
an \((n-1)\)\-computad, a set of \(n\)\-generators and an attaching function
\(V_n^C\to \type_{n-1}(C_{n-1})\), assigning to each generator a source and
target. A morphism of computads \(f\colon C\to D\) is a morphism between the
\(\omega\)\-categories generated by the computads, and it is defined to be a
pair \((f_{n-1}. f_V)\) of a morphism \(C_{n-1} \to D_{n-1}\) together with a
function \(V_n^C\to \cell_n(D)\) preserving the source and target, in that the
following diagram commutes
\[\begin{tikzcd}[column sep = huge]
  V_n^C & \cell_n(D) \\
  \type_{n-1}(C_{n-1}) & \type_{n-1}(D_{n-1})
  \arrow[from=1-1, to=1-2, "f_V"]
  \arrow[from=2-1, to=2-2, "\type_{n-1}(f_{n-1})"]
  \arrow[from=1-1, to=2-1, "\phi_n^C"']
  \arrow[from=1-2, to=2-2, "\ty_{n,D}"]
\end{tikzcd}\]
The truncation functor is the first projection on both objects and
morphisms.

The set \(\cell_n C\) is defined inductively by two rules. The first states that
every generator \(v\in V_n^C\) gives rise to a cell that we denote by
\(\var c\). The second rule states that there exists a cell
\(\coh{B}{A}{f}\) for every Batanin tree \(B\) of dimension at most \(n\), every
full \(A\in \Full_{n-1}(B)\) and every morphism
\(f \colon \free_n \Pos B\to C\). The boundaries of those cells are defined
recursively by
\begin{align*}
  \ty_n(\var v) &= \phi_n^C(v) &
  \ty_n(\coh{B}{A}{f}) &= \type_{n-1}(f_{n-1})(A)
\end{align*}
Notice that the definition of morphisms uses cells and vice versa. This
apparent circularity is resolved using induction-recursion, i.e. reading
the definitions above as the description the initial algebra for some polynomial
endofunctor~\cite[Section~3.3]{dean_computads_2024}.

Composition with a morphism \(f \colon C\to D\) and its action on cells are
defined mutually recursively by
\begin{gather*}
  \cell(f)(\var v) = f_V(v)\\
  \cell(f)(\coh{B}{A}{g}) = \coh{B}{A}{f\circ g} \\
  f\circ g = (f_{n-1}\circ g_{n-1}, \cell_{n-1}(f_{n-1})\circ g_{n-1})
\end{gather*}
together with a proof that \(\ty_n\) is natural. Using mutual induction, we can
then show that this composition operation is associative and unital, and that
\(\cell_n\) is a functor.

The free functor \(\free_n\) sends a globular set \(X\) to the \(n\)\-computad
consisting of \(\free_{n-1}X\), the set \(X_n\) and the attaching function
given by
\[
  \phi_n^{\free X}(x) = (\var \src x , \var \tgt x).
\]
It sends a morphism of globular sets \(f \colon X\to Y\) to the morphism of
\(n\)\-computads consisting of \(\free_{n-1}f\) and the function
\(\var\circ f_n\). The functor of \(n\)\-spheres together with the projection
natural transformations are defined via the following pullback square
\[
  \begin{tikzcd}
    \type_n \rar{\pr_1}\dar[swap]{\pr_2}\pbmark & \cell_n \dar{\ty_n} \\
    \cell_n \rar[swap]{\ty_n} & \type_{n-1}u_n
  \end{tikzcd}
\]
that is, an \(n\)\-sphere is a pair of \(n\)\-cells with the same boundary. We
will often denote such a pair \((a,b)\in\type_n(C)\) by \(\mor ab\), and
write \(c \colon \mor ab\) to denote that a cell \(c\in \cell_{n+1}(C)\) has
boundary \((a,b)\).

To conclude the induction, it remains to define when an \(n\)\-sphere of a
Batanin tree \(B\) is full. To do that, we define the \emph{support}
of a cell \(c\in \cell_n(C)\) to be the set of generators appearing in \(c\).
More formally, the support of \(c\) is defined recursively by
\begin{align*}
  \supp_n(\var v) &= \{v\} &
  \supp_n(\coh{B}{A}{f}) = \bigcup_{p\in \Pos[n]{B}} \supp_n(f_V(p))
\end{align*}
and then a sphere \((a,b)\in \type_n(\free_n(\Pos B))\) is declared to be
full when the support of \(a\) consists of the \(n\)\-positions in the image of
\(s_n^B\) and the support of \(b\) consists of the \(n\)\-positions in the image
of \(t_n^B\). Moreover, if \(n > 0\), we require that the \((n-1)\)\-sphere
\(\ty_n(a) = \ty_n(b)\) is also full. Full spheres amount to ways to compose
the boundary of a globular pasting diagram when \(n = \dim B - 1\), and they
amount to pairs of parallel ways to compose the whole diagram when
\(n \ge \dim B\).

The category of computads is then defined to be the limit of the forgetful
functors \(u_n\), so a computad \(C = (C_n)_{n\in\N}\) is a sequence of
\(n\)\-computads such that \(u_{n+1}C_{n+1} = C_n\). The free functors
\(\free_n \colon \glob \to \comp_n\) are compatible with the forgetful functors,
so they give rise to a functor
\[
  \free\colon \glob\to \comp.
\]
In the opposite direction, we may define a functor
\[
  \cell \colon \comp \to \glob
\]
sending a computad \(C\) to the globular set with cells given by
\(\cell_n(C_n)\) and with source and target functions given by the composition
of the boundary natural transformation \(\ty_n\) with the projections \(\pr_i\).
The functor \(\free\) is left adjoint to \(\cell\). The unit \(\eta\) of the
adjunction is given by the morphisms of globular sets
\[
  \eta_X \colon X \to \cell \free (X)
\]
sending \(x\in X_n\) to the generator cell \(\var x\). The counit
\(\varepsilon\) consists of the morphisms of computads
\[
  \varepsilon_C \colon \free\cell C \to C
\]
determined by the identity functions
\[
  V_n^{\free\cell C} = \cell_n C
\]
The triangle equations for this adjunction can be easily checked.

\begin{remark}\label{rmk:alt-fullness}
  The inductive fullness condition above is equivalent to the following one, as
  shown by Dean et al.~\cite{dean_computads_2024}. We can define more generally
  for a cell \(c\in \cell_n(C)\) its \(k\)\-support \(\supp_{k}(c)\) to be the
  set of \(k\)\-dimensional generators used in the definition of \(c\), or its
  source and target. We will say that \(c\) covers \(C\) when its support
  contains every generator of \(C\). We then say that a sphere
  \(A = (a,b)\in \type_n(\free \Pos{B})\) is full if and only if the support of
  \(a\) is the image of \(s_n^B\) and that of \(b\) is the image of \(t_n^B\).
  This is equivalent in turn to \(a = T(s_n^B)(a')\) and \(b = T(t_n^B)(b')\)
  for cells \(a',b'\in (T\Pos{\bdry[n] B})_n\) that cover
  \(\free \Pos{\bdry[n]{B}}\).
\end{remark}

\subsection{\texorpdfstring{Operations in \(\omega\)}{ω}-categories}
\label{subsec:omega-cats}

Having defined the adjunction between computads and globular sets, we may now
define \(\omega\)\-categories. This definition is equivalent to the one of
Leinster, as shown by Dean et al.~\cite{dean_computads_2024}.

\begin{definition}
  The \emph{free weak \(\omega\)-category monad} \((T, \eta, \mu)\) is the monad
  induced by the adjunction \(\free \dashv \cell\). The category \(\Cat_\omega\)
  of (weak) \(\omega\)-categories is the category of \(T\)\-algebras.
\end{definition}

\noindent
By definition, an \(\omega\)\-category is a pair
\(\wcat X = (X, \alpha \colon TX\to X)\) satisfying an associativity and a unit
axiom. In particular, for every Batanin tree \(B\) and every cell
\(c\in (T\Pos B)_n\), there exists an operation
\begin{gather*}
  c^{\wcat X} \colon \glob(\Pos B, X) \to X_n \\
  c^{\wcat X}(f) = (\alpha\circ Tf)(c)
\end{gather*}
that is natural in that for every morphism of \(\omega\)\-categories
\(g\colon \wcat X \to \wcat Y\),
\[
  g(c^\wcat X(f)) = c^{\wcat Y}(Ug\circ f).
\]
It was shown by the second author~\cite{markakis_computads_2024} that those
operations, and more specifically the operations of the form
\[\coh[\wcat X]{B}{A}{-} = \coh{B}{A}{\id}^{\wcat X}\]
fully determine the structure morphism \(\alpha\) and that they can be chosen
freely subject to source and target conditions. We will call such operations
\emph{compositions} when \(\dim B = n\) and \emph{coherences} when
\(\dim B < n\).

Utilising the idea that natural operations in an \(\omega\)\-category with
arity the pasting diagram \(\Pos B\) correspond to elements of
\(T\Pos B\), to define composition and identity operations in an
\(\omega\)\-category, we will construct a family of cells over pasting
diagrams. More precisely, we define recursively for every Batanin tree
\(B\) and every natural number \(n\), a full \(n\)\-sphere
\(A_{B,n} \in \Full_n(B)\), and we define a cell \(\unbcomp_{n,B}\) with
boundary \(A_{B,n-1}\) when \(n \ge \dim B\) by
\begin{align*}
  \unbcomp_{B,n} &= \begin{cases}
    \var(\id_n) &\text{when } B = D_n \\
    \coh{B}{A_{B,n-1}}{\id} &\text{otherwise}
  \end{cases} \\
  A_{B,n} &= (T(s_n^B)(\unbcomp_{\partial_n B, n}),
    T(t_n^B)(\unbcomp_{\partial_n B, n}))
\end{align*}
The \emph{identity} of an \(\omega\)\-category \(\wcat X\) is the operation
\[
  \id^{\wcat X}_n =
  \unbcomp_{D_n, n+1}^{\wcat X} \colon X_n\cong \glob(\Pos{D_n}, X) \to X_{n+1}
\]
taking a cell \(x\in X_n\) to a cell with source and target \(x\). The
\emph{unbiased composition} over a tree \(B\) is the operation
\[
  \unbcomp^{\wcat X}_B = \unbcomp_{B,\dim B}^{\wcat X} \colon
  \glob(\Pos B, X) \to X_{\dim B},
\]
taking a diagram \(f \colon \Pos B\to X\) to a cell with source and target
the unbiased composite of \(f\circ s^B\) and \(f\circ t^B\) respectively. In
particular, for the Batanin tree
\(B = \graft{k_m}{\graft{k_1}{D_{n_0}}{\dots}}{D_{n_m}}\),
a morphism \(f\colon \Pos{B}\to X\) amounts to a sequence of cells
\(x_i \in X_{n_i}\) such that \(x_i\) and \(x_{i+1}\) are
\(k_{i+1}\)\-composable, and we will write the unbiased composite of \(f\) as
\[\unbcomp^{\wcat X}_{B}(f) =
  \compcell{k_{1}}{x_{0}}{\compcell{k_{m}}{\ldots}{x_{m}}}\]
omitting the index \(k_i\) when \(\dim x_i = \dim x_{i+1} = k+1\).

\subsection{Suspensions and opposites}
\label{subsec:suspension}

Constructing high-dimensional operations of \(\omega\)\-category, or
equivalently cells over a globular pasting diagram, tends to be a
difficult task, leading us to introduce meta-operations that produce new such
cells from existing ones. In our previous paper~\cite{benjamin_opposites_2024},
we introduced two such meta-operations, the \emph{suspension} and the
\emph{opposite}. The former is obtained by interpreting an operation in the
hom \(\omega\)\-categories of an \(\omega\)\-category, while the latter is
obtained by interpreting it in its opposites, introduced in the same paper.

More formally, we define mutually recursively the suspension of a computad
\(\susp \colon \comp\to \comp\) extending the one for globular sets, together
with a natural transformation
\[\begin{tikzcd}
  \glob \ar[r, "\free"] \ar[d, "\susp"'] &
  \comp \ar[r, "\cell"] \ar[d, "\susp"'] &
  \glob\ar[d, "\susp"] \\
  \glob \ar[r, "\free"] &
  \comp \ar[r, "\cell"] &
  \glob
  \ar[from = 1-2, to = 2-1, equals, shorten >= 4pt, shorten <= 4pt,]
  \ar[from = 1-3, to = 2-2, Rightarrow, "\susp^{\cell}"',
    shorten >= 4pt, shorten <= 4pt,]
\end{tikzcd}\]
by the following recursive formulae
\begin{align*}
  (\susp C)_0 &= \set{v_-,v_+} &
  (\susp C)_{n+1}
    &=((\susp C)_n, V_n^C , (\susp^{\cell}, \susp^{\cell})\circ\phi_n^C) \\
  (\susp f)_0 &= \id &
  (\susp f)_{n+1} &= ((\susp f)_n, \susp^{\cell}\circ \circ f_{V,n}) \\
  \susp^{\cell}(\var v) &= \var v &
  \susp^{\cell}(\coh{B}{A}{f})
    &= \coh{\susp B}{(\susp^{\cell}, \susp^{\cell})A}{\susp f}
\end{align*}
The suspension operation allows us to take an operation in the form of a cell
\(c\in (T\Pos B)_n\) and produce a new operation
\(\susp^{\cell}(c)\in (T\Pos{\susp B})_{n+1}\) of higher dimension. For example,
suspending the unbiased composition over the tree
\(\Chain_k = \graft{0}{D_1}{\graft{0}{\cdots}{D_1}}\), we obtain the unbiased
composition operation over the tree
\(\susp^n\Chain_k = \graft{n}{D_{n+1}}{\graft{n}{\cdots}{D_{n+1}}}\), which
corresponds to the composition of \(k\) consecutive \((n+1)\)\-cells.

To define the opposite of a computad, we proceed similarly. Given a set of
positive natural numbers \(w\in \N_{>0}\), we may define an autoequivalence
\[
  \op \colon \glob\to \glob
\]
by swapping the source and target of every \(n\)\-cell of a globular set when
\(n\in w\). This operation preserves globular pasting diagrams, in the sense
that there exists an automorphism \(\op \colon \bat\to \bat\) together with
a family of isomorphisms \(\op^B \colon \Pos{\op B}\to \op \Pos{B}\) for every
Batanin tree \(B\) compatible with the source and target inclusions. We then
extend \(\op\) to an autoequivalence of the category of computads
\(\op\colon \comp\to \comp\) together with a natural isomorphism
\[\begin{tikzcd}
  \glob \ar[r, "\free"] \ar[d, "\op"'] &
  \comp \ar[r, "\cell"] \ar[d, "\op"'] &
  \glob\ar[d, "\op"] \\
  \glob \ar[r, "\free"] &
  \comp \ar[r, "\cell"] &
  \glob
  \ar[from = 1-2, to = 2-1, equals, shorten >= 4pt, shorten <= 4pt,]
  \ar[from = 1-3, to = 2-2, Rightarrow, "\op^{\cell}"',
    shorten >= 4pt, shorten <= 4pt,]
\end{tikzcd}\]
by similar recursive formulae as above:
\begin{align*}
  (\op C)_n &= ((\op C)_{n-1}, V_n^C,
    \swap_n \circ (\op^{\cell},\op^{\cell})\phi_n^C) \\
  (\op f)_n &= ((\op f)_{n-1}, \op^{\cell} \circ f_{V,n}) \\
  \op^{\cell}(\var v) &= \var v \\
  \op^{\cell}(\coh{B}{A}{f}) &= \coh{\op B}{A'}{(\op f)\circ \free_n(\op^B)} \\
  A' &= (T(\op_w^B)^{-1}, T(\op_w^B)^{-1})
    \circ \swap_n \circ (\op^{\cell},\op^{\cell})A
\end{align*}
where \(\swap_n\) swaps the two components of the pullback when \(n+1\in w\) and
it is the identity otherwise. The opposite operation allows us to take an
operation, e.g. an unbiased composition, and construct a new operation over the
pasting diagram with all the cells reversed. For example, the
\(\set{1}\)\-opposite of the pasting diagram \(B = \graft{0}{D_1}{D_2}\) is the
pasting diagram \(\op B = \graft{0}{D_2}{D_1}\). The unbiased composition over
\(B\) is the left whiskering of an \(1\)\-cell with an \(2\)\-cell, and it is
sent to the unbiased composite over \(\op B\) which is the right whiskering.

\subsection{Invertible and equivalent cells}\label{subsec:invertible-cell}

A crucial notion in the study of higher categories is that of equivalence. This
is usually defined by induction on the dimension of the structure: two
elements of a \(0\)\-category are equivalent when they are equal, while two
objects \(x,y\) in an \((n+1)\)\-category are equivalent when there exist
morphisms \(f \colon x\to y \) and \(g\colon y\to x\) such that the compositions
\(f\circ g\) and \(g\circ f\) are equivalent to identities in the respective
\(n\)\-categories of morphisms. In such case, we say that the morphisms \(f\)
and \(g\) are invertible. This definition of equivalence fails may be also
used for \(\omega\)\-categories if interpreted coinductively, as observed
by Cheng~\cite{cheng_omega_2007}:

\begin{definition}\label{def:invertibility}
  The collection of \emph{invertible} cells of an \(\omega\)\-category
  \(\wcat X\) is defined coinductively by saying that a positive-dimensional
  cell \(x \colon u\to v \in X_{n+1}\) is invertible if there exists a cell
  \(\inv x \colon v\to u\), together with a pair of invertible cells
  \begin{align*}
    \unit_x &\colon x \ast_{n} \inv x\to \id_n(u) &
    \counit_x &\colon \inv x \ast_{n} x\to \id_n(v).
  \end{align*}
  We say that a pair of cells \(c, c'\in X_n\) are \emph{equivalent} and write
  \(c \sim c'\) when there exists an invertible cell with source \(c\) and
  target \(c'\).
\end{definition}

\noindent
The general semantics of coinduction is beyond the scope of this article. To
explain the definition, let \(\mathbf{X} = \sqcup_{n}X_{n+1}\) the set of all
positive-dimensional cells of \(\wcat X\), and consider the endofunction \(F\)
sending a subset \(U \subset \mathbf{X}\) to the set of cells
\(x\in \mathbf{X}\) for which there exist \(\inv x \in \mathbf{X}\) and
\(\unit_{x},\counit_{x}\in U\) satisfying the same boundary conditions as in
Definition~\ref{def:invertibility}. The function \(F\) is monotone, so by
Tarski's fixed point theorem, it has a greatest postfixed point \(W\), which we
will call the set of invertible cells of \(\wcat X\). By definition, it is the
maximum set such that \(W \subset F(W)\). This provides a method for proving
invertibility of a set of cells \(U\): it suffices to show that
\(U\subseteq F(U)\) to conclude that \(U\subseteq W\). We will use this
method to show for example that coherence operations produce invertible cells.

\begin{lemma}\label{lem:equiv-sym-and-functors}
  The relation \(\sim\) is symmetric and preserved by every morphism.
\end{lemma}
\begin{proof}
  Let \(\wcat X\) be an \(\omega\)\-category and \(c\sim c'\) be equivalent
  cells. Then there exists an invertible cell \(x\) with source \(c\) and
  target \(c'\). The inverse \(\inv x\) is again invertible with
  \(\inv{(\inv x)} = x\), \(\unit_{\inv x} = \counit_x\) and
  \(\counit_{\inv x} = \unit_x\). Its source is \(c'\) and its target is \(c\),
  so \(c' \sim c\). Therefore, the relation \(\sim\) is symmetric.

  To show that a morphism \(f \colon \wcat X\to \wcat Y\) preserves equivalence,
  we will show coinductively that it preserves invertibility. More precisely,
  let \(U\) the set of cells of \(\wcat Y\) that are the image of some
  invertible cell of \(\wcat X\). Then for every cell \(y\in U\), there exists
  some invertible cell \(x\) of \(\wcat X\) such that \(f(x) = y\). Then we
  may define \(\inv y = f(\inv x)\) and \(\unit_y = f(\unit_x) \in U\) and
  \(\counit_y = f(\counit_x)\in U\). This shows that \(U\subset F(U)\), so
  every cell of \(U\) is invertible in \(\wcat Y\).
\end{proof}

\begin{prop}\label{prop:invertible-coh}
  Let \(B\) be a Batanin tree and \(A \in \Full_{n-1}(B)\) be a full
  sphere such that \(\dim B < n\). Then for every \(\omega\)\-category
  \(\wcat X\) and every morphism of globular sets \(f\colon \Pos B\to X\), the
  coherence cell \(\coh[\wcat X]{B}{A}{f}\) is invertible.
\end{prop}
\begin{proof}
  We will first show coinductively that for a Batanin tree \(B\), the set of
  cells \(U\) of the form \(\coh{B}{A}{\id}\) for some full sphere
  \(A\in \Full_{n-1}(B)\) such that \(\dim B < n\) consists of invertible
  cells. For that, let \(x = \coh{B}{A}{\id} \in X_n\) be such a cell and let
  \(A = (u,v)\). By the assumption on the dimension of \(A\), the \(n\)\-spheres
  \((u,u)\), \((v,v)\) and \((v,u)\) are all full, so we may define the
  inverse cell \(\inv x = \coh{B}{(v,u)}{\id}\). The support of the cells
  \(x\ast_{n-1} \inv x\), \(\id_n(u)\), \(\inv x\ast_{n-1} x\) and \(\id_n(v)\)
  are empty, since \(B\) has no positions of dimension \(n\), and their
  boundaries are full as explained above. Therefore, we may also define the
  cells
  \begin{align*}
    \unit_x &= \coh{B}{\mor {x\ast_{n-1} \inv x}{\id_n(u)}}{\id} \in U \\
    \counit_x &= \coh{B}{\mor{\inv x\ast_{n-1} x}{\id_n(v)}}{\id} \in U.
  \end{align*}
  By coinduction, it follows that every cell in \(U\) is invertible.

  Let now \(\wcat X = (X,\alpha)\) be an \(\omega\)\-category and
  \(f\colon \Pos B\to X\) be a morphism of globular sets. Then
  \(Tf \colon T\Pos B\to TX\) is a morphism of free \(\omega\)\-categories
  \(F\Pos B\to FX\), and \(\alpha \colon TX\to X\) is a morphism
  \(FX\to \wcat X\), hence the
  cell
  \[
    \coh[\wcat X]{B}{A}{f} = (\alpha\circ Tf)(\coh{B}{A}{\id})
  \]
  is the image of an invertible cell under a morphism of \(\omega\)\-categories.
  It follows by Lemma~\ref{lem:equiv-sym-and-functors} that
  \(\coh[\wcat X]{B}{A}{f}\) is invertible.
\end{proof}

\begin{corollary}\label{cor:equiv-refl}
  The relation \(\sim\) is reflexive.
\end{corollary}

\begin{proof}
  The corollary follows by invertibility of identity cells.
\end{proof}



\section{Constructions in weak \texorpdfstring{\(\omega\)}{ω}-categories}
\label{sec:constructions}
This section is dedicated to expanding our toolbox to work with
\(\omega\)\-categories. We extend the language of computads with additional
constructions, defining simpler ways to describe some cells.

\subsection{Unbiased unitors}
We first define a family of cells that we call the \emph{unbiased unitors}.
Intuitively, those are cells which take composite of identities onto an
identity. More formally, a composite of identities is a composite where all
top-dimensional cells are identities. In order to define those precisely, we
rely on the following result.

\begin{lemma}\label{lemma:id-everywhere}
  For a Batanin tree \(B\) and every \(n\in\N\), the source and target inclusions
  \(s^{B}_n, t^{B}_n : \Pos{\bdry[n] B} \to \Pos{B}\) induce bijections between
  positions of dimension \(k < n\). Moreover, they are injective on positions of
  dimension \(n\), and for every position \(p \in \Pos[n]{B}\), there exists
  unique position \(q\in \Pos[n]{B}\) such that \(s_n^B(q)\), \(t_n^B(q)\) and
  \(p\) are parallel.
\end{lemma}
\begin{proof}
  The proof is by straightforward induction on the tree and on \(n\). It can
  be seen also from the pictorial description of trees and their positions.
\end{proof}

\noindent
Consider a Batanin tree \(B\) of dimension \(d\). Using
Lemma~\ref{lemma:id-everywhere} we define the map
\(\underline{\id} : \free\Pos{B} \to \free{\Pos{\bdry[d-1] B}}\),
defined recursively as follows:
\begin{itemize}
  \item To each position \(p\) of dimension \(k < d-1\), it assigns the cell
    \(\var q\), where \(q\) is the preimage of \(p\)
  \item To each position \(p\) of dimension \(d-1\), it assigns the cell
    \(\var q\), where \(q\) is the unique position such that \(s_{d-1}^B(q)\)
    and \(p\) are parallel
  \item To each position \(p\) of dimension \(d\), it assigns the cell
    \(\id(\underline{\id}(\src p))\)
\end{itemize}
By construction, it follows that \(\underline{\id}\) is a common retraction of
the source and target inclusions:
\[
  \underline{\id} \circ \free(s_{d-1}^B) =
  \underline{\id} \circ \free(t_{d-1}^B) =
  \id_{\free \Pos{\bdry B}}
\]

\begin{definition}
  Consider a Batanin tree \(B\) of dimension \(d\) together with a cell
  \(a \in (T\Pos{\bdry[d-1] B})_{d-1}\) that covers \(\bdry[d-1]{B}\).
  The \emph{unbiased unitor}
  \(\unitor(B,a)\) is the cell
  \[
    \unitor(B,a) = \coh{\bdry B}{u\to v}{\id_{\Pos{\bdry[d-1] B}}}\in
    (T\Pos{\partial B})_d
  \]
  where
  \begin{align*}
    u &= \coh{B}{T(s_{d-1}^{B})(a)\to T(t_{d-1}^{B})(a)}{\underline{\id}} &
    v &= \id(a)
  \end{align*}
\end{definition}

\noindent The assumption on the support of \(a\) implies that both \(u\) and
\(v\) are well-defined cells, while the two cells are parallel by the
computation of the source and target of \(\underline{\id}\). The cell acts as
a unitor, since it takes a composite over \(B\) where all top dimensional
cells are sent to identities to the identity of the composite.

\subsection{Filler cells}
We define a collection of coherence operations, that we call fillers,
generalising the associators,
unitors and interchangers. Those are operations in \(\omega\)\-categories, that
we describe again using cells over a globular pasting diagram. Recall that given
a diagram of Batanin trees \(f\colon \Pos{B}\to \bat\) indexed by a Batanin
tree \(B\), we may form the composite tree \(\mu^{\str}f\).

\begin{definition}
  Suppose that for each \(i \in \set{1,2}\), we are given
  \begin{itemize}
    \item a Batanin tree \(B_i\),
    \item morphisms \(f_{i} : \Pos{B_{i}} \to \bat\) such that
      \(\mu^{\str}(f_{1}) = \mu^{\str}(f_{2})\),
    \item full spheres \(A_i\) in \(\Pos{B_{i}}\)
    \item covering morphisms
      \(\sigma_{i} : \free \Pos{B_{i}} \to \free\Pos{\mu^{\str}(f_{i})}\) such
    that
    \begin{align*}
      \type(\sigma_{1})(A_{1}) &= \type(\sigma_{2})(A_{2})
    \end{align*}
  \end{itemize}
  we define the filler cell
  \[
    \filler(B_{i},f_{i},a_{i}\to b_{i},\sigma_{i}) = \coh{\mu^{\str}(f_{1})}
    {c_{1} \to c_{2}}{\id}
  \]
  where
  \begin{align*}
    c_{1} &= \coh{B_{1}}{A_1}{\sigma_{1}}
    & c_{2} &= \coh{B_{2}}{A_2}{\sigma_{2}}.
  \end{align*}
\end{definition}

\noindent The assumptions on \(A_i\) and \(\sigma_i\) ensure exactly that the
cells \(c_{1}\) and \(c_{2}\) are cell defined, parallel and covering the tree
\(\mu^{\str}(f_{i})\). By Proposition~\ref{prop:invertible-coh}, all the fillers
are all invertible cells.

\begin{example}
  The associator with source is three binary composed \(1\)\-cells, associated
  on the left and its target is three binary composed \(1\)\-cells, associated
  on the right can be obtained as a filler with:
  \begin{itemize}
  \item The tree \(B_{1} = \compcell{0}{D_{1}}{D_{1}}\), the map
    \(f_{1} = \pomap{\compcell{0}{D_{1}}{D_{1}}, D_{1}}\), the cells
    \(a_{1} = \var p\) and \(b_{1} = \var q\) where \(p,q\) are respectively the
    left- and right-most \(0\)\-positions of \(B_{1}\), and the morphism
    \(\sigma_{1} = \pomap{\unbcomp_{\compcell{0}{D_{1}}{D_{1}}},\id_{D_{1}}}\),
  \item The tree \(B_{2} = \compcell{0}{D_{1}}{D_{1}}\), the map
    \(f_{2} = \pomap{D_{1},\compcell{0}{D_{1}}{D_{1}}}\), the cells
    \(a_{2} = \var p\) and \(b_{2} = \var q\) where \(p,q\) are respectively the
    left- and right-most \(0\)\-positions of \(B_{2}\), and morphism
    \(\sigma_{2} = \pomap{\id_{D_{1}},\unbcomp_{\compcell{0}{D_{1}}{D_{1}}}}\).
  \end{itemize}
\end{example}
\begin{example}
  The unbiased unitor \(\unit(B,a)\) can also be obtained as a filler, by
  choosing:
  \begin{itemize}
  \item The tree \(B_{1} = B\), the map \(f_{1}\) associating to every position
    \(p\) the tree \(D_{\min(\dim p, \dim B -1)}\), the cells
    \(a_{1} = T(s^{B}_{d-1})(a)\) and \(a_{2} = T(t^{B}_{d-1})(a)\), and the map
    \(\sigma_{1} = \underline{\id}\).
  \item The tree \(B_{2} = D_{\dim B - 1}\), the map \(f_{2}\) associating to
    every position \(p\) the tree \(\bdry[\dim p] B\), the cells
    \(a_{2} = b_{2} = \var p\) where \(p\) is the unique maximal position of
    \(D_{\dim B - 1}\), and the map \(\sigma_{2}\) corresponding to the cell
    \(a\) via the Yoneda lemma.
  \end{itemize}
\end{example}

\subsection{Functorialisation of coherences}
Given a Batanin tree \(B\) of dimension \(d\) together with a set
\(X \subseteq \Pos[d]{B}\) of maximal positions of \(B\)
(c.f.~Definition~\ref{def:lmax-pos}), we define the functorialisation of the
tree \(B\) with respect to the set \(X\) by induction on \(B\), denoted
\(\fun{X}{B}\) as follows
\begin{align*}
  \fun{\emptyset}{B} &= B \\
  \fun{X}{(\bt{})}
  &= \bt{\bt{}}\\
  \fun{X}{(\bt{B_{1},\ldots,B_{n}})}
  &= \bt{\fun{(\inc_{1}^{-1}(X))}{B_{1}},\ldots,\fun{(\inc_{n}^{-1}(X))}{B_{n}}}
\end{align*}
Intuitively, this operation consists in selecting a set of leaves of the tree,
and growing one more branch on top of all the selected leaves. For instance,
consider the tree displayed on the left. The functorialisation of this tree with
respect to the set of locally maximal positions indicated as red produces
the tree represented on the right. The newly created branches are also displayed
in red, to improve legibility.
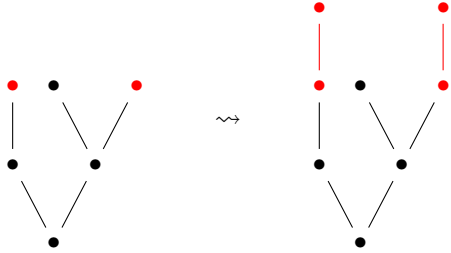
\[
  \begin{tikzcd}[column sep = 0.2em]
    \phantom{\bullet}\\
    {\color{red}\bullet} & \bullet &  & {\color{red} \bullet}  \\
    \bullet \ar[u,-]& & \bullet \ar[ul,-]\ar[ur,-] & \\
                         & \bullet \ar[ul,-]\ar[ur,-] & &
  \end{tikzcd}
  \qquad\rightsquigarrow\qquad
  \begin{tikzcd}[column sep = 0.2em]
    {\color{red} \bullet} & & & {\color{red} \bullet} \\
    {\color{red} \bullet} \ar[u,-,red]
                          & \bullet &  & {\color{red} \bullet} \ar[u,-,red]  \\
     \bullet\ar[u,-] & & \bullet \ar[ul,-]\ar[ur,-] & \\
    & \bullet \ar[ul,-]\ar[ur,-] & &
  \end{tikzcd}
\]
Consider for instance the Batanin tree
\(\Chain_{k} = \graft{0}{D_{1}}{\graft{0}{\ldots}{D_{1}}}\) and the position
\(f^{\Chain}_{i}\). Then the functorialisation is given by
\[
  \fun{f^{\Chain}_{i}}{\Chain_{k}} =
  \graft{0}{D_{1}}{\graft{0}{\ldots}{\graft{0}{D_{1}}
      {\graft{0}{D_{2}}{\graft{0}{D_{1}}{\graft{0}{\ldots}{D_{1}}}}}}},
\]
where in this expression, the disk \(D_{2}\) appears in the \(i\)-th position.

\begin{lemma}
  For a Batanin tree \(B\) of dimension \(d\), and a non-empty set \(X\) of
  maximal positions in \(B\), the dimension of the functorialisation is given by
  \(\dim(\fun{X}{B}) = d + 1\), and we have \(\bdry[d](\fun{X}{B}) = B\).
\end{lemma}
\begin{proof}
  We proceed by induction on the tree \(B\). If \(B = \bt{}\), then \(\dim B = 0\) and
  \(\dim ({\fun{X}{B}}) = \dim (D_{1}) = 1\), and moreover,
  \(\bdry[0]({\fun{X}{B}}) = B\). Otherwise, \(B = \bt{B_{1},\ldots,B_{n}}\) for \(n>0\), and then
  \(\dim B = 1+\max{(\dim B_{i})}\), while
  \[\dim(\fun{X}{B}) = 1+\max{(\dim(\fun{\inc_{i}^{-1}(X)}{B_{i}}))}.\] By
  induction, we have \(\dim(\fun{\inc_{i}^{-1}(X)}{B_{i}}) = 1 + \dim B_{i}\) if
  \(\inc_{i}^{-1}(X)\) is non-empty, and \(\dim B_{i}\) otherwise. Moreover,
  since \(X\) contains at least one maximal position of \(B\) there exists at
  least one \(i\) such that \(\inc_{i}^{-1}(X)\). We then have necessarily
  \(1 + \dim B_{i} = \dim B\), and thus \(\dim (\fun{X}{B}) = 1 + \dim B\).
  Moreover, if we let \(\dim B = d\), we have that
  \[\bdry[d](\fun{X}{B}) = \bt{\bdry[d-1](\fun{\inc_{1}^{-1}(X)}{B_{1}}),\ldots,
    \bdry[d-1](\fun{\inc_{n}^{-1}(X)}{B_{n}})}.\]
  If \(\inc_{i}^{-1}(X)\neq \emptyset\), then \(\dim B_i = d\) and by the
  induction, we have that \[\bdry[d-1](\fun{\inc_{i}^{-1}(X)}{B_{i}}) = B_{i}.\]
  Otherwise, we have that
  \(\fun{\inc_{i}^{-1}(X)}{B_{i}} = B_{i}\) is of dimension at most \(d-1\),
  from which the same equality holds. Hence, \(\bdry[d]{B} = B\).
\end{proof}
\begin{defi}
  Given a Batanin tree \(B\) with a set \(X\) of maximal positions, we
  define the globular set \(\Pos{B}\setminus X\) by letting
  \((\Pos{B}\setminus X)_{n} = \Pos[n]{B}\setminus X\), with the source and
  target maps the restriction of those of \(\Pos{B}\).
\end{defi}

\begin{lemma}\label{lemma:functorialisation-po}
  Given a Batanin tree \(B\) of dimension \(d\) with a set \(X\) of maximal
  positions, there exists a colimit cocone of the form
  \[
    \begin{tikzcd}
      &&  \Pos{B}\ar[ddr, "s^{\fun{X}{B}}_{d}"] \\
      & \coprod\limits_{x\in X} \disk{d}\ar[ur,"\pomap{x}"']
        \ar[dr,"\pomap{s^{D_{d+1}}_{d}}"] \\
      \Pos{B}\setminus X \ar[rruu, bend left, hook]\ar[rrdd, bend right, hook]
      && \coprod\limits_{x\in X} \disk{d+1}\ar[r] & \Pos{\fun{X}{B}} \\
      & \coprod\limits_{x\in X} \disk{d}\ar[dr,"\pomap{x}"]
        \ar[ur,"\pomap{t^{D_{d+1}}_{d}}"'] \\
      &&  \Pos{B}\ar[uur, "t^{\fun{X}{B}}_{d}"'] \\
    \end{tikzcd}
  \]
  The unnamed morphism picks the generators corresponding to the newly grown
  branches of the functorialised tree.
\end{lemma}

The proof of the lemma is by a straightforward induction on the tree \(B\). This
gives a method for constructing a map \(\sigma : \Pos{\fun{X}{B}} \to Y\): Such
a map amounts to a pair of maps
\(\sigma_{-},\sigma_{+} : \Pos{B} \to Y\) which coincide on
\(\Pos{B}\setminus X\) as well as, for every \(x\in X\), a cell
\(f_{x} \in Y_{n+1}\) such that \(\src (f_{x}) = \sigma_{-}(x)\) and
\(\tgt(f_{x}) = \sigma_{+}(x)\).

\begin{remark}
  Lemma~\ref{lemma:functorialisation-po} gives a characterisation of the
  functorialisation without relying on the Batanin tree, hinting at a more
  general functorialisation operation valid on every globular set, or even every
  computad with respect to well-chosen generators. This construction has been
  studied by the first author in the setting of the type theory
  \catt~\cite{benjamin_type_2020}, and will not be useful for the purpose of
  this article.
\end{remark}

Having defined the functorialisation of a Batanin tree, we may define the
functorialisation of an operation \(\coh(B,A,\id)\) in a straightforward manner.
This is a special instance of the functorialisation of a cell, described by
the first author~\cite{benjamin_type_2020}.

\begin{definition}\label{def:functorialisaiton-coh}
  Given a Batanin tree \(B\) of dimension \(d\) together with a set of maximal
  positions \(X \subset \Pos{B}\), and two parallel cells
  \(a,b\in T\Pos{\bdry[d-1]{B}}\) that cover \(\bdry[d-1]{B}\), we define the
  functorialisation
  \[
    \fun{X}{(B,a\to b)} = \coh{\fun{X}{B}} {T(s^{\fun{X}{B}}_{d})(c) \to
      T(t^{\fun{X}{B}}_{d})(c)}{\id_{\Pos{\fun{X}{B}}}}
  \]
  where \(c = \coh{B}{a\to b}{\id_{\free \Pos{B}}}\) is covering \(T\Pos B\).
\end{definition}

\subsection{Chain reduction}

Before introducing our final operation on trees, that we call
\emph{chain reduction}, we need to introduce a family of trees that we call
\emph{chains}. Those are easily defined for \(k\in\N\) as the composites
\(\Chain_k = \graft{0}{D_1}{\graft{0}{\cdots}{D_1}}\), which correspond to the
composition of \(k\) consecutive \(1\)\-cells. Explicitly the globular set of
positions of \(\Chain_k\) is the following diagram
\[
  \begin{tikzcd}
    x^{\Chain}_{0} \ar[r,"f^{\Chain}_{1}"]& x^{\Chain}_{1} \ar[r,"f^{\Chain}_{2}"]
    & \ldots \ar[r,"f^{\Chain}_{k}"]& x^{\Chain}_{k}
  \end{tikzcd}
\]
We note that chains are precisely the trees of dimension at most \(1\). We can
then get higher dimensional chains \(\susp^n\Chain_d\) by using the suspension
operation. The tree
\(\susp^n\Chain_d = \graft{n}{D_{n+1}}{\graft{n}{\cdots}{D_{n+1}}}\) is the
tree consisting of \(k\) consecutive \((n+1)\)\-cells.

The \emph{chain reduction} of a Batanin tree \(B\) of dimension \(d\) is a new
tree \(\red{B}\) obtained by merging the chains of positions of dimension \(d\)
in \(B\) into a unique position. It is defined together with a morphism
\[\reduced_B\colon \free \Pos{\red{B}}\to \free\Pos{B}\] that sends the merged
position into the composite of the original chain. To define the chain reduction
of a tree, we define more generally for \(i \ge \dim B - 1\) a new tree
\(\red[i]{B}\) recursively as follows
\begin{align*}
  \red[i]{\bt{}} & = \bt{} \\
  \red[0]{\Chain_{k}} & = \bt{\bt{}} \\
  \red[i+1]{\bt{B_{1},\ldots,B_{n}}} & = \bt{\red[i]{B_{1}},\ldots,\red[i]{B_{n}}}.
\end{align*}
In these definitions, the first case takes precedence over the second one. It
follows that \(\red[i]{B} = B\) for \(i \ge \dim B\), so that the only newly
introduced operation is \(\red[\dim B-1]{B}\) which we will denote simply by
\(\red{B}\). For the morphism of computads \(\reduced_B\), we define again a
morphisms \(\reduced_{i,B} \colon \free\Pos{\red[i]{B}} \to \free\Pos{B}\)
recursively as follows:
\begin{align*}
  \reduced_{i,\bt{}}
  & = \id_{\free\disk{0}}\\
  \reduced_{0,\Chain_{k}}
  &= \unbcomp_{\Chain_{k}} \\
  \reduced_{i+1,\bt{B_{1},\ldots,B_{n}}}
  &= \pomap{\susp\reduced_{i,B_{1}},\ldots,\susp\reduced_{i,B_{n}}}.
\end{align*}
Those morphisms are again identities for \(i\ge \dim B\), and we denote simply
\(\reduced_{\dim B - 1,B}\) by \(\reduced_B\). Here,
\(\unbcomp_{\Chain_{k}}\) is the morphism \(\free \disk{1} \to \free C_{k}\),
which via the adjunction \(\free \dashv \cell\) corresponds to the map
\(\disk{1} \to TC_{k}\) given by the Yoneda lemma on the cell with the same
name. We note also that the last
case uses the equality \(\susp\free = \free\susp\) proven
in~\cite{benjamin_opposites_2024}, as well as the fact that \(\free\) is left
adjoint and thus preserves colimits.

\begin{lemma}\label{lemma:chain-reduce-map}
  Consider a Batanin tree \(B\) of dimension \(d\). Then, for every maximal
  position \(p\) of \(\red{B}\), there exists a natural number
  \(\length_{B}(p)\) and a map \(\chain_{B,p}\) satisfying the following:
  \begin{align*}
    \chain_{B,p}
    &\colon \Pos{\susp^{d-1}\Chain_{\length_{B}(p)}} \to \Pos{B}\\
    \reduced_{B}(p)
    &= T(\chain_{B,p})(\unbcomp_{\susp^{d-1}\Chain_{\length_{B} (p)}}).
  \end{align*}
\end{lemma}
\begin{proof}
  For the sake of simplicity, we often omit the index \(B\) in \(\length\) and
  \(\chain\). We prove this result by induction on \(d\): If \(d \le 1\), then by
  definition \(B = \Chain_{k}\) for unique \(k\), so we may let
  \(\length(p) = k\) and \(\chain_{p} = \id_{\Pos{B}}\).
  So it suffices to prove the result for
  \(d > 1\), assuming it holds for trees of dimension lower than \(d\). In this
  case, we write \(B = \bt{B_{1},\ldots,B_{l}}\). By induction, considering
  \(j\) such that \(\dim(B_{j}) = \dim(B)-1\), for every maximal position
  \(p_{j}\) of \(B_{j}\), we get an number \(\length_{B_{j}}(p_{j}) \in\N_{>0}\)
  and a map
  \begin{align*}
    \chain_{p_{j}}
    &\colon \free \Pos{\susp^{d-2}\Chain_{\length(p_{j})}} \to \free\Pos{B_{j}}\\
    \reduced_{B}(p_{j})
    &= T(\chain_{p_{j}})(\unbcomp_{\susp^{d-2}\Chain_{\length (p_{j})}}).
  \end{align*}
  Consider the canonical map \(\inc_{j} : \susp\Pos{B_{i}}\to \Pos{B}\), then
  for \(p = \inc_{j}(\susp p_{j})\), we let
  \(\length_{B}(p) = \length_{B_{j}}(p_{j})\) together with the map
  \(\chain_{p}\) to be the following composite
  \[
    \chain_{p} = \inc_{j} \circ \susp\chain_{p_{j}}.
  \]
  For this definition, we use the equality
  \(\susp\circ\PosOp = \PosOp\circ\susp\) (c.f.~\cite{benjamin_opposites_2024}).
  We have the following equation:
  \begin{align*}
    \reduced_{B}(p)
    &=
      \pomap{\susp\reduced_{i,B_{1}},\ldots,\susp\reduced_{i,B_{n}}}
      (\inc_{j}(\susp p_{j})) \\
    &= T(\inc_{j})\susp^{\cell}(\reduced_{i,B_{j}}(p_{j})) \\
    &= T(\inc_{j}\circ (\susp \chain_{p_{j}}))
      (\unbcomp_{\susp^{d-1}\Chain_{\length(p_{j})}})\\
    &= T(\chain_{p})(\unbcomp_{\susp^{d-1}\Chain_{\length(p)}}). \tag*{\qedhere}
  \end{align*}
\end{proof}

\noindent
Taking the reduction of a Batanin tree does not change its boundary. This
equality is compatible with the source and target inclusion of the boundary in
the pasting scheme, as stated by the following result.
\begin{lemma}
  \label{lemma:bdry-reduction}
  Given a Batanin tree \(B\) of dimension \(d\), we
  have the equality \(\bdry[d-1]\red{B} = \bdry B\) and the following
  squares commute
  \begin{align*}
    \begin{tikzcd}[ampersand replacement = \&]
      \free\Pos{\bdry[d-1](\red{B})}
      \ar[d,equal]\ar[r,"\free s^{\red{B}}_{d-1}"]
      \& \free\Pos{\red{B}} \ar[d,"\reduced_{B}"]\\
      \free \Pos{\bdry[d-1] B}\ar[r,"\free s^{B}_{d-1}"] \& \free\Pos{B}
    \end{tikzcd}
    \\
    \begin{tikzcd}[ampersand replacement = \&]
      \free\Pos{\bdry[d-1](\red{B})}
      \ar[d,equal]\ar[r,"\free t^{\red{B}}_{d-1}"]
      \& \free\Pos{\red{B}} \ar[d,"\reduced_{B}"]\\
      \free \Pos{\bdry[d-1] B}\ar[r,"\free t^{B}_{d-1}"] \& \free\Pos{B}
    \end{tikzcd}
  \end{align*}
\end{lemma}
\begin{proof}
  We will show by induction that \(\partial_i \red[i]{B} = \red[i]{B}\) for all
  \(i\ge \dim B-1\), and the corresponding diagrams commute. The case where
  \(i \ge \dim B\) is trivial with both sides equal to \(B\), so we let
  \(i = \dim B -1\). The case where \(B = \bt{}\) is then vacuously true. In the
  case where \(B = \Chain_{k}\) for some \(k\) and \(i = 0\), we have
  \(\bdry[0]{\Chain_{k}} = \bt{} = \bdry[0]{\bt{\bt{}}}\). Moreover, we have
  \[
    \reduced_{0,\Chain_{k}}\circ \free(s^{D_{1}}_{0})
    = \src(\unbcomp_{\Chain_{k}})
    = \free(s^{\Chain_{k}}_{0}).
  \]
  Finally, for \(i > 0\), we have \(B = \bt{B_{1},\ldots,B_{n}}\), and by
  induction, \(\bdry[i-1]{B_{k}} = \bdry[i-1]{\red[i-1]B_{k}}\), which imply the
  equality \(\bdry[i]{B} = \bdry[i]{\red[i]{B}}\). Moreover, we have that
  \begin{align*}
    \reduced_{i,B}\circ &\free(s^{\red[i]{B}}_{i}) \\
    &= \pomap{\susp(\reduced_{i-1,B_{1}}\circ \free(s^{\red[i-1]B_{1}}_{i-1})),
      \ldots,\susp\reduced_{i-1,B_{n}} \free(s^{\red[i-1]B_{n}}_{i-1})} \\
    &= \pomap{\susp \free(s^{B_{1}}_{i-1}),\ldots \susp \free(s^{B_{n}}_{i-1})}
    \\
    &= \free(s^{B}_{i}) \tag*{\qedhere}
  \end{align*}
\end{proof}

\noindent
If two Batanin trees \(B\) and \(B'\) have the same chain reduction, then they
must be of the same dimension and composable. The following lemma identifies the
reduction of their composite, as well as the chain maps.

\begin{lemma}\label{lemma:red-comp}
  Given two Batanin trees \(B\), \(B'\) of dimension \(d\) such that
  \(\red{B} = \red{B'}\), we have \(\red{(\compcell{d-1}{B}{B'})} = \red{B}\).
  Moreover, for every maximal position \(p \in \Pos[d]{\red{B}}\), we have that
  \[\length_{\graft{d-1}{B}{B'}} (p) = \length_{B}(p) + \length_{B'}(p)\]
  and that the following diagrams commute:
  \begin{align*}
    \begin{tikzcd}[ampersand replacement = \&]
      \susp^{d-1}\Pos{\Chain_{\length_{B}(p)}}
      \ar[r,"\chain_{B,p}"]
      \ar[d,hook,"\susp^{d-1} \inc^-"']
      \& \Pos{B}\ar[d,"\inc^-"] \\
      \susp^{d-1}\Pos{\Chain_{\length_{\graft{d-1}{B}{B'}}(p)}}\ar[r,"\chain_{p}"']
      \&\Pos{\graft{d-1}{B}{B'}}
    \end{tikzcd}
    \\
    \begin{tikzcd}[ampersand replacement = \&]
      \susp^{d-1}\Pos{\Chain_{\length_{B'}(p)}}
      \ar[r,"\chain_{B',p}"]
      \ar[d,hook,"\susp^{d-1} \inc^+"']
      \& \Pos{B}\ar[d,"\inc^+"] \\
      \susp^{d-1}\Pos{\Chain_{\length_{\graft{d-1}{B}{B'}}(p)}}\ar[r,"\chain_{p}"']
      \&\Pos{\graft{d-1}{B}{B'}}
    \end{tikzcd}
  \end{align*}
  for \(\inc^\pm\) the pushout inclusions associated to
  \(\Chain_{k+l} = \graft{0}{\Chain_{k}}{\Chain_{l}}\).
\end{lemma}
\begin{proof}
  We proceed by induction on the dimension of \(B\) and \(B'\). If both trees
  are of dimension at most \(1\), then the identity
  \(\Chain_{k+l} = \graft{0}{\Chain_{k}}{\Chain_{l}}\) implies the result on
  length, the maps \(\chain_{p}\) are identity maps. If \(B\) and \(B'\) are of
  dimension \(d+1\), then we write \(B = \bt{B_{1},\ldots,B_{n}}\) and
  \(B' = \bt{B'_{1},\ldots,B'_{m}}\). The condition \(\red{B} = \red{B'}\)
  implies that \(m = n\) and that for every \(i\),
  \(\red{B_{i}} = \red{B'_{i}}\). Then, given a maximal position
  \(p\in \Pos{\red{B}}\), there exists \(p_{i} \in \Pos{\red{B_{i}}}\) such that
  \(\inc_{i}(p_{i}) = p\). We then have by induction
  \begin{align*}
    \length_{\graft{d}{B}{B'}}(p)
    &= \length_{\graft{d+1}{B_{i}}{B'_{i}}}(p_{i})\\
    &= \length_{B_{i}}(p_{i}) + \length_{B'_{i}}(p_{i}) \\
    &= \length_{B}(p) + \length_{B'}(p).
  \end{align*}
  Moreover, we can prove the first square by induction as follows:
  \begin{align*}
    \chain_{\graft{d}{B}{B'},p} \circ\susp^{d}{\inc^-}
    &= \inc_{i} \circ\susp(\chain_{\graft{d}{B_{i}}{B'_{i}},p_{i}}
      \circ \susp^{d-1}{\inc^-})\\
    &= \inc_{i} \circ\susp(\inc^-\circ \chain_{B_{i},p_{i}})\\
    &= \inc^-\circ\inc_{i}\circ \susp \chain_{B_{i},p_{i}} \\
    &= \inc^- \circ \chain_{B,p}.
  \end{align*}
  The commutativity of the other square is similar.
\end{proof}

Using those lemmas, we may now define the chain reduction of a composite cell
in a computad. Suppose that a computad \(C\) is given together with a cell
\[
  c = \coh{B}{A}{\sigma} \in \cell_d(C)
\]
where \(d = \dim B\). As explained in Remark~\ref{rmk:alt-fullness} that implies
that \[A = T(s^{B}_{d-1})(a) \to T(t^{B}_{d-1})(b)\]
for some covering cells \(a,b\in (T\Pos{\partial_{d-1}B})_{d-1}\).
Using Lemma~\ref{lemma:bdry-reduction}, we define the chain reduction of \(c\)
to be
\[
  \red{c} = \coh{\red{B}}{s^{\red{B}}_{d-1}(a) \to
    t^{\red{B}}_{d-1}(b)}{\sigma\circ\reduced_{B}} \in \cell(C).
\]
This cell defines a particular biassing of \(c\), in the sense that it is a cell
which is weakly equivalent to the cell \(c\).
This can be seen by constructing a filler
\[
  \assoc(c) = \filler \left(B_{i},f_{i},a_{i} \to b_{i}, \sigma_{i}\right)
  \colon c\to \red{c},
\]
where \(B_{1} = B\), \(f_{1}\) is the map associating to \(p\) the disc
\(D^{\dim p}\), \(a_{1} = T(s^{B}_{d-1})(a)\), \(b_{1} = T(t^{B}_{d-1})(b)\) and
\(\sigma_{1} = \sigma\), while \(B_{2} = \red{B}\), \(f_{2}\) is the map
associating to every maximal position \(p\) the tree
\(\susp^{d-1}\Chain_{\length (p)}\), \(a_{2} = T(s^{\red B}_{d-1})(a)\),
\(b_{2} = T(t^{\red B}_{d-1})(b)\) and \(\sigma_{2} = \sigma\circ\reduced_{B}\).
This filler is an invertible cell whose source is \(c\) and whose target is
\(\red{c}\).



\section{Composite of invertible cells}
\label{sec:inverse-composite}
Our aim in this section is to show that a cell \(c\) obtained as a composite of
invertible cells is invertible. We achieve this by constructing explicitly the
inverse \(\inv{c}\) and the witnesses \(\unit_{c}, \counit_{c}\). Our strategy
for constructing these cells can be illustrated on a simple example: Consider
the following composite cell \(c\) displayed on the left-hand side (with
\(a,b,c\) invertible in \(\wcat X\)), the inverse that we construct is displayed
on the right hand-side
\begin{align*}
  c &=
  \begin{tikzcd}[ampersand replacement = \&]
    \bullet \ar[r,bend left=50]\ar[r,bend right=50]\ar[r, bend left = 25,
    phantom, "\Downarrow_{a}"]\ar[r, bend right = 25,
    phantom, "\Downarrow_{b}"]\ar[r]
    \& \bullet\ar[r, bend left]\ar[r, bend right]\ar[r,phantom,"\Downarrow_{c}"]
    \& \bullet
  \end{tikzcd}
  & \inv{c} &=
  \begin{tikzcd}[ampersand replacement = \&]
    \bullet \ar[r,bend left=50]\ar[r,bend right=50]\ar[r, bend left = 25,
    phantom, "\Downarrow_{\inv{b}}"]\ar[r, bend right = 25,
    phantom, "\Downarrow_{\inv{a}}"]\ar[r]
    \& \bullet\ar[r, bend left]\ar[r, bend right]
    \ar[r,phantom,"\Downarrow_{\inv{c}}"]
    \& \bullet
  \end{tikzcd}
\end{align*}
We note that the order of \(a\) and \(b\) must be swapped, so \(\inv c\) must be
a composite over the opposite pasting diagram. Defining the cell \(\unit_{x}\)
requires composing
\(\unit_{a}\), \(\unit_{b}\) and \(\unit_{c}\) together in order to cancel each
of the cells with its inverse. This can be done in multiple ways, as long as
\(b\) is cancelled before \(a\). In order to get a systematic scheme, we define
the cancellation of the composite of \(a\)
and \(b\) with the composite of \(b\) and \(a\) as an intermediate step, called
a \emph{telescope} below, and we
then perform the cancellation of this composite in parallel with the
cancellation of \(c\) in an unbiased way. In general, we cancel the
composition of maximal cells in codimension \(1\) first and then proceed in
an unbiased way.

\subsection{Pointwise inverse of a morphism}

Out of a morphism out of a globular pasting diagram that sends maximal positions
to invertible cells, we may derive a new morphism out of the opposite pasting
diagram by inverting the images of the maximal positions. This new morphism will
be useful for constructing the inverse of a composite of invertible cells, as
illustrated in the example at the beginning of this section. Existence and
uniqueness of such a morphism is guaranteed by the following lemma.

\begin{lemma}
  \label{lemma:inverse-comp-sub}
  Let \(X\) be a globular set, \(d > 0\) a positive natural number and \(B\) a
  Batanin tree of dimension at most \(d\). Then \(\op_{\set{d}}B = B\), and
  for every morphism of globular sets \(\sigma \colon \Pos{B}\to X\)
  and every family of cells \(\inv{\sigma(p)}\),
  indexed by \(p\in \Pos[d]{B}\), whose source and target are those of
  \(\sigma(p)\) swapped, there exists a well-defined morphism of globular sets
  \(\bar{\sigma} \colon \Pos{B}\to X\) sending a position
  \(q\in \Pos[k]{B}\) to
  \[
    \bar{\sigma}(q) = \begin{cases}
      \inv{\sigma(\op^{B}_{\set{d}}(q))} & \text{if } k = d, \\
      \sigma(\op^{B}_{\set{d}}(q)) & \text{if } k < d,
    \end{cases}
  \]
  where \(\op^B_{\set{d}}\colon
  \Pos{B}= \Pos{\op_{\set{d}} B} \to \op_{\set{d}}\Pos{B}\)
  is the isomorphism described in Section~\ref{subsec:suspension}. Moreover,
  the source and target of \(\sigma\) and \(\bar{\sigma}\) are swapped, in that
  \begin{align*}
    \bar{\sigma} \circ s_{d-1}^B &= \sigma \circ t_{d-1}^B &
    \bar{\sigma} \circ t_{d-1}^B &= \sigma \circ s_{d-1}^B
  \end{align*}
\end{lemma}
\begin{proof}
  We proceed by induction on the Batanin tree \(B = \bt{B_1,\dots,B_n}\).
  Suppose first that \(d = 1\), then by the hypothesis on the dimension, it
  must be the case that \(B_i = \bt{}\) for every \(i\), so that
  \[
    \op_{\set{1}}B = \bt{B_n,\dots,B_1} = B.
  \]
  A morphism \(\sigma \colon \Pos{B}\to X\) amounts to a sequence of consecutive
  \(1\)\-cells \(f_1,\dots,f_n\) of \(X\) when \(n > 0\), or to a \(0\)\-cell
  otherwise. The morphism \(\bar{\sigma}\) corresponding to
  the consecutive cells \(\inv f_n, \dots, \inv f_1\) when \(n > 0\), or the
  chosen \(0\)\-cell otherwise, satisfies the claimed formula by definition of
  the automorphism \(\op_{\set{1}}^B\).

  Suppose now that \(d > 1\). Then every tree \(B_i\) has dimension at most
  \(d-1\) and by the inductive hypothesis, we get that
  \[
    \op_{\set{d}}B = \bt{\op_{\set{d-1}}B_1, \dots,\op_{\set{d-1}}B_n}
    = \bt{B_1,\dots,B_n} = B.
  \]
  By the universal property of the wedge sum, a morphism
  \(\sigma \colon \Pos{B}\to X\) can be written as a wedge sum of morphisms
  \(\sigma_i \colon \Sigma \Pos{B_i}\to X\). Transposing those morphisms,
  we get morphisms \(\sigma_i^\dagger \colon \Pos{B_i}\to \Omega X\) for the
  appropriate choices of basepoints for \(X\) depending on \(i\). By the
  inductive hypothesis applied to the morphisms
  \(\sigma_i^\dagger\) and the family of cells
  \(\inv{\sigma_i(p)} = \inv{\sigma(\inc_i(p))},\) indexed by
  \(p\in\Pos[d-1]{B_i}\), there exist morphisms
  \(\bar{\sigma}_i \colon \Pos{B_i}\to \Omega X\) satisfying the formula of the
  lemma. Finally, transposing those morphisms again and wedging them, we get a
  morphism
  \[
    \bar{\sigma} =
    \bigvee_{i=1}^n\bar{\sigma}_i^\dagger\colon
    \Pos{B} \to X
  \]
  In the case that \(d > 0\), we have that the opposite of the wedge sum is
  equal to the wedge sum of the opposites, and similarly with the suspension,
  so the isomorphism \(\op_{\set{d}}^B\) is given simply by
  \[\op_{\set{d}}^B = \bigvee_{i=1}^n \susp \op_{\set{d-1}}^{B_i}.\]
  It follows that the constructed morphism \(\bar{\sigma}\) is given by the
  claimed formula.

  Finally, it remains to show that the source and target of the morphisms
  \(\sigma\) and \(\bar{\sigma}\) are swapped. To show that, we observe first
  that for every tree \(B\) and for every \(d > \dim B\), the globular sets
  \(\op_{\set{d}}\Pos{B}\) and \(\Pos{B}\) are equal, since the latter has no
  positions of dimension \(d\). Moreover, by induction on the tree, we can
  further show that the automorphism \(\op_{\set{d}}^B\) is the identity. It
  follows in particular that for arbitrary tree \(B\) and arbitrary \(d > 0\),
  the isomorphism \(\op_{\set{d}}^{\bdry[d-1]{B}}\) is an identity. Combining
  this fact with the defining formula of \(\bar{\sigma}\), and the compatibility
  of the isomorphism \(\op^B_{\set{d}}\) with the source and target morphisms,
  established in our previous work~\cite[Lemma~7]{benjamin_opposites_2024}, the
  claim follows.
\end{proof}

\subsection{Telescopes}

We define the \emph{telescopes}, a family of operations that allows us to cancel
sequences of consecutive cells composed with their inverses. Again,
using the suspension, it suffices to define those operations
for the composite of \(k\) consecutive \(1\)\-cells in codimension \(0\).
Contrary to previously defined operations such as identities and composites,
the telescopes have arities computads
that are not globular pasting diagrams. In general, such cells can be thought of
as ``proof tactics''. Here, the tactics
we are describing consist of the following steps: associate
the two cells in the middle of a composite of even arity, rewrite their
composition into an identity, cancel the identity, and repeat inductively.

We start by defining the \(2\)\-computad \(\Tel_{k}\) for \(k \in \N_{>0}\),
which are the smallest computads in which the telescope cells are defined.
Those computads can be visualised as
\[
  \Tel_{k} =
  \begin{tikzcd}
    x_{0}
    \ar[r, bend left, "f_{1}"]
    \ar[r, phantom, "\overset{\alpha_{1}}{\Leftarrow}"]
    &
      \ar[l, bend left, "g_{1}"]
      x_{1}
      \ar[r, bend left, "f_{2}"]
      \ar[r, phantom, "\overset{\alpha_{2}}{\Leftarrow}"]
    &
      \ar[l, bend left, "g_{2}"]
      x_{2}
      \ar[r, bend left, "f_{3}"]
      \ar[r, phantom, "\overset{\alpha_{3}}{\Leftarrow}"]
    &
       \ar[l, bend left, "g_{3}"]
      \cdots
      \ar[r, bend left, "f_{k}"]
      \ar[r, phantom, "\overset{\alpha_{k}}{\Leftarrow}"]
    &
       \ar[l, bend left, "g_{k}"]
      x_{k}
  \end{tikzcd}
\]
where \(\alpha_i \colon \graft{0}{f_i}{g_i} \to \id\). More formally, the
computad \(\Tel_k\) is determined by
\begin{align*}
  V^{\Tel_k}_0 &= \set{x^{\Tel}_0,\dots,x^{\Tel}_k}
  & \phi^{\Tel_k}_1(f^{\Tel}_i) &= x^{\Tel}_i\to x^{\Tel}_{i+1} \\
  V^{\Tel_k}_1 &= \set{f^{\Tel}_1,g^{\Tel}_1\dots,f^{\Tel}_k,g^{\Tel}_k}
  & \phi^{\Tel_k}_1(g_i^{\Tel}) &= x^{\Tel}_{i+1}\to x^{\Tel}_i \\
  V^{\Tel_k}_2 &= \set{\alpha^{\Tel}_1,\dots,\alpha^{\Tel}_k}
  & \phi^{\Tel_k}_2(\alpha^{\Tel}_i) &= \graft{0}{f^{\Tel}_i}{g^{\Tel}_i} \to
                                       \id (x^{\Tel}_{i-1})
\end{align*}
and by \(V_n^{\Tel_k} = \emptyset\) for \(k > 2\), where the generators
\(x^{\Tel}_i,f^{\Tel}_i,g^{\Tel}_i,\alpha^{\Tel}_i\) are drawn from disjoint countable sets.
The inclusion of the generators of \(\Tel_{k}\) into generators of
\(\Tel_{k+1}\) induces a generator-preserving monomorphism of computads, that we
denote by \(\telinc_{k} : \Tel_{k} \hookrightarrow \Tel_{k+1}\). Moreover,
the \(1\)\-generators of the computad \(\Tel_{k}\) can be
composed: we may first define the morphism of computads
\(\loopmap_{k} : \free \Pos{\Chain_{2k}} \to \Tel_{k}\), characterised by
\[
  \loopmap_{k}(f^{\Chain}_{i}) =
  \begin{cases}
    \var f^{\Tel}_{i} & \text{If \(1 \leq i \leq k\)} \\
    \var g^{\Tel}_{2k-i+1} & \text{If \(k+1 \leq i \leq 2k\)}
  \end{cases}
\]
This lets us construct the cell
\(\cell(\loopmap_{k})(\unbcomp_{\Chain_{2k}}) \in \cell_1 (\Tel_{k})\). One can
check that it is a cell whose source and target are both given by
\(\var x^{\Tel}_{0}\).

\begin{theorem}
  There exists a cell \(\tel_{k}\in\cell(\Tel_{k})\) whose
  source is the composite \(\cell(\loopmap_{k})(\unbcomp_{\Chain_{2k}})\) and
  whose target is the identity \(\idcell{0}{x^{\Tel}_{0}}\).
\end{theorem}
\begin{proof}
  We construct the cell \(\tel_{k}\) by induction on \(k\), starting from
  \(\tel_0\) being the identity. We choose then
  \(\tel_{1} = \var \alpha^{\Tel}_{1}\), and we define \(\tel_{k+1}\) as the
  composite of four cells, corresponding to the four steps in the proof tactics
  that \(\tel_{k+1}\) encodes:
\begin{align*}
  \tel_{k+1} =
  & \compcell{1}{\tel_{k+1}^{1}}
    {\compcell{1}{\tel_{k+1}^{2}}
    {\compcell{1}{\tel_{k+1}^{3}}
    {\cell(\telinc_{k})(\tel_{k})}}}.
\end{align*}
It suffices to define the cells \(\tel_{k+1}^{1}\), \(\tel_{k+1}^{2}\) and
\(\tel_{k+1}^{3}\). The cell \(\tel_{k+1}^{1}\) is the application of an
associator that takes the unbiased composite of the boundary of the telescope to
a composite where the cells \(\var f^{\Tel}_{k+1}\) and \(\var g^{\Tel}_{k+1}\)
are associated together. To define this cell, we note that there is a morphism
of computads
\(\sigma\colon\free\Pos{\Chain_{2k+1}} \to \free\Pos{\Chain_{2k+2}}\),
defined by
\[
  \sigma_{V}(f^{\Chain}_{i}) =
  \begin{cases}
    \var f^{\Chain}_{i} & \text{if \(i < k+1\)} \\
    \compcell{0}{(\var f^{\Chain}_{k+1})}{(\var f^{\Chain}_{k+2})}
                        & \text{if \(i = k+1\)}\\
    \var f^{\Chain}_{i+1} & \text{if \(i > k+1\)}
  \end{cases}
\]
We then define the first cell of the telescope to be the following cell, given
here with its source and target:
\begin{align*}
  \tel_{k+1}^{1} &= \coh{\Chain_{2k+2}} {\unbcomp_{\Chain_{2k+2}} \to
                   \cell(\sigma)(\unbcomp_{\Chain_{2k+1}})}
                   {\loopmap_{k+1}} \\
  &
    \begin{aligned}
      \src(\tel_{k+1}^1) &= \compcell{0}{f^{\Tel}_{1}}
                           {\compcell{0}{\ldots}{\compcell{0}{f^{\Tel}_{k+1}}
                           {\compcell{0}{g^{\Tel}_{k+1}}
                           {\compcell{0}{\ldots}{g^{\Tel}_{1}}}}}}\\
      \tgt(\tel_{k+1}^1) &= \compcell{0}{f^{\Tel}_{1}}
                           {\compcell{0}{\ldots}{\compcell{0}{f^{\Tel}_{k}}
                           {\compcell{0}{(\compcell{0}{f^{\Tel}_{k+1}}{g^{\Tel}_{k+1}})}
                           {\compcell{0}{g^{\Tel}_{k}}
                           {\compcell{0}{\ldots}{g^{\Tel}_{1}}}}}}}.
    \end{aligned}
\end{align*}
The second cell consists in using the generator \(\alpha^{\Tel}_{k+1}\) in order
to relate \(\compcell{0}{f^{\Tel}_{k+1}}{g^{\Tel}_{k+1}}\) to the identity on
\(x^{\Tel}_{k}\). We define the cell as follows, given here with its source and
target:
\begin{align*}
  \tel_{k+1}^{2}
  &= \cell(\pomap{f^{\Tel}_{1},\ldots,f^{\Tel}_{k},\alpha^{\Tel}_{k+1},
    g^{\Tel}_{k},\ldots,f^{\Tel}_{1}})
    ((\unbcomp_{\fun{f^{\Chain}_{k+1}}{(\Chain_{2k+1})}})) \\
  & \begin{aligned}
    \src(\tel_{k+1}^{2}) &= \compcell{0}{f^{\Tel}_{1}}
                           {\compcell{0}{\ldots}{\compcell{0}{f^{\Tel}_{k}}
                           {\compcell{0}{(\compcell{0}{f^{\Tel}_{k+1}}{g^{\Tel}_{k+1}})}
                           {\compcell{0}{g^{\Tel}_{k}}
                           {\compcell{0}{\ldots}{g^{\Tel}_{1}}}}}}}\\
    \tgt(\tel_{k+1}^{2}) &= \compcell{0}{f^{\Tel}_{1}}
                           {\compcell{0}{\ldots}{\compcell{0}{f^{\Tel}_{k}}
                           {\compcell{0}{\id (x^{\Tel}_{k})}
                           {\compcell{0}{g^{\Tel}_{k}}
                           {\compcell{0}{\ldots}{g^{\Tel}_{1}}}}}}}.
  \end{aligned}
\end{align*}
The last cell removes the identity in the middle of the composite. We define a
map \(\tau : \free\Pos{\Chain_{2k+1}} \to \free\Pos{\Chain_{2k}}\), by the
following assignment:
\[
  \tau(f^{\Chain}_{i}) =
  \begin{cases}
    \var f^{\Chain}_{i} & \text{if \(i < k+1 \)} \\
    \id(\var x^{\Chain}_{k}) & \text{if \(i = k+1\)} \\
    \var f^{\Chain}_{i-1} & \text{if \(i > k+1\).}
  \end{cases}
\]
This lets us define the last of the three cells as follows, given again with
its source and target:
\begin{align*}
  \tel_{k+1}^{3}
  &= \coh{\Chain_{2k}}
    {\cell(\tau)(\unbcomp_{\Chain_{2k+1}}) \to \unbcomp_{\Chain_{2k}}}
    {\loopmap_{k}}\\
  & \begin{aligned}
    \src(\tel_{k+1}^{3}) &= \compcell{0}{f^{\Tel}_{1}}
                           {\compcell{0}{\ldots}{\compcell{0}{f^{\Tel}_{k}}
                           {\compcell{0}{\id (x^{\Tel}_{k})}
                           {\compcell{0}{g^{\Tel}_{k}}
                           {\compcell{0}{\ldots}{g^{\Tel}_{1}}}}}}}\\
    \tgt(\tel_{k+1}^{3}) &= \compcell{0}{f^{\Tel}_{1}}
                           {\compcell{0}{\ldots}{\compcell{0}{f^{\Tel}_{k}}
                           {\compcell{0}{g^{\Tel}_{k}}
                           {\compcell{0}{\ldots}{g^{\Tel}_{1}}}}}}.
  \end{aligned}
\end{align*}
This finishes the definition of the cell \(\tel_{k+1}\).
\end{proof}

\subsection{Invertibility of composition}
From now on, we consider an \(\omega\)-category
\(\wcat X = (X, \alpha : TX \to X)\). Our aim is to show that a composite of
invertible cells in \(\wcat X\) is itself invertible, we start by making this
statement precise. Given a set of cells \(A\) of \(\mathbf{X} = \sqcup X_{n}\),
we define the set \(C(A)\) of composites of cells of \(A\) to be the set of all
cells of the form
\[\coh[\wcat X]{B}{T(s^{B}_{d-1})(a)\to T(t^{B}_{d-1})(b)}{\sigma},\] where
\(B\) is a Batanin tree of dimension \(d\), the cells
\(a,b \in (T\Pos{\bdry[d-1]B})\) are cover \(\bdry[d-1]{B}\) and
\(\sigma : \Pos{B} \to X\) is a map sending maximal positions of \(B\) onto
cells in \(A\). Given the set \(W\) of invertible cells in \(\wcat X\), the set
of composites of invertible cells is the set \(C(W)\), and we prove here that
every cell in \(C(W)\) is invertible. For the coinductive hypothesis, we need a
slightly stronger statement, namely that iterated composited of invertible cells
are invertible. To set up the notation, define
\begin{align*}
  \left\{\begin{aligned}
    W_{0} &= W \\
    W_{n+1} &= C\left( W_{\leq k}\right)
  \end{aligned}\right. &
  & W_{\leq n} &= \bigcup\limits_{k\leq n} W_{k}
  & W_{\infty} &= \bigcup\limits_{n} W_{n}.
\end{align*}
The set \(W_{\infty}\) is the set of iterated composites of invertible cells.

\begin{thm}\label{thm:inv-comp}
  All cells in \(W_{\infty}\) are invertible. In other words,
  \(W_{\infty} \subset W\).
\end{thm}
\begin{proof}
  By coinduction, it suffices to show that given a cell \(c \in W_{\infty}\),
  there exists cells \(\inv c \in \mathbf{X}\) and
  \(\unit_{c},\counit_{c}\in W_{\infty}\) satisfying the correct boundary
  conditions. Indeed, this translates into \(W_{\infty}\) being a postfixed
  point of the function defining \(W\), thus it implies that
  \(W_{\infty} \subset W\). Since \(W_\infty\) is a union of sets, it suffices
  to prove for every \(n\in\N\) that for every \(c\in W_n\), there exist
  \(\inv c \in \mathbf{X}\) and \(\unit_{c},\counit_{c}\in W_{\infty}\)
  satisfying the correct boundary conditions. We shall prove this statement by
  strong induction on \(n\in\N\). The base case \(n = 0\) being true by
  definition of invertibility. The rest of the proof is devoted to the inductive
  step. For that, we fix a number \(n\in \N\) and a cell \(c\in W_{n+1}\). By
  definition of \(W_{n+1}\) and Remark~\ref{rmk:alt-fullness}, we may write
  \(c \in X_d\) in the form
  \[
    c = \coh[\wcat X]{B}{T(s^{B}_{d-1})(a) \to T(t^{B}_{d-1})(b)}{\sigma}
  \]
  for some tree \(B\) is a Batanin tree of dimension \(d\), a pair of cells
  \(a,b \in T\Pos{\bdry[d-1]B}\) that cover \(\bdry[d-1]B\), and a morphism
  \(\sigma\colon \Pos{B}\to X\) sending all maximal positions of \(B\)
  to cells in \(W_{\leq n}\).

  To simplify the notation, we will drop the subindices from
  \(\graft{d-1}{B}{B}\) and from \(\op_{\set{d}}\). We will also introduce the
  cells
  \begin{align*}
    c_{0} = \coh{B}{T(s^{B}_{d-1})(a) \to T(t^{B}_{d-1})(b)}{\id_{\free\Pos{B}}}. \\
    c_0' = \coh{B}{T(s^{B}_{d-1})(b) \to T(t^{B}_{d-1})(a)}{\id_{\free\Pos{B}}}.
  \end{align*}
  The former is a cell satisfying that \(c = c_0^{\wcat X}(\sigma)\). By the
  inductive hypothesis, for any maximal position \(p \in \Pos[d]{B}\),
  we have constructed a cell \(\inv{\sigma_{d,V}(p)} \in W_{\le n}\) with the
  opposite source and target than \(\sigma_{d,V}(p)\). Therefore, we may define
  \begin{align*}
    \inv c &= (c'_0)^{\wcat X}(\bar{\sigma}) =
    \coh[\wcat X]{B}{T(s^{B}_{d-1})(b) \to T(t^{B}_{d-1})(a)}{\bar\sigma},
  \end{align*}
  where \(\bar{\sigma} : \Pos{B} \to X\) is the map constructed by
  Lemma~\ref{lemma:inverse-comp-sub}. By the same lemma, we can see that
  \(\inv c\) has the appropriate boundary.

  To conclude the proof, we will construct the cell \(\unit_c\). The cell
  \(\counit_c\) can be obtained in the same way by replacing \(c_0\) with the
  cell \(c'_0\) and the morphism \(\sigma\) with \(\bar{\sigma}\).
  Alternatively, we may assume by induction on \(n\) that
  \(\inv c \in W_{n+1}\). In this case, the constructed inverse
  \(\inv{(\inv c)}\) is equal to the original cell \(c\), so long as we choose
  the inverses of the cells in \(W_n\) to have the same property. Therefore,
  constructing the cell \(\unit_c\) for every \(c\in W_{n+1}\) suffices,
  since we may then define \(\counit_c = \unit_{\inv c}\).

  The cell
  \(\unit_{c} \in W_{\infty}\) will be defined as a composite of three different
  cells
  \[
    \unit_{c} = \compcell{n+1}{m^{1}}{\compcell{n+1}{m^{2}}{m^{3}}}
  \]
  a generalised associator, a composite of telescopes and a generalised unitor.
  The associator and the unitor will both be invertible by
  Proposition~\ref{prop:invertible-coh}. Each of the telescopes will be a cell
  in \(W_\infty\), as explained below, so that \(\unit_c\in W_\infty\).
  Moreover, by construction, we will have that
  \begin{align*}
    \src(\unit_{c}) &= \src(m_1) = \graft{n}{c}{\inv c}
    & \tgt(\unit_{c}) &= \tgt(m_3) = \id(\src c),
  \end{align*}
  so that \(\unit_c\) has the correct boundary.

  \paragraph{Associator.} The first step in constructing \(\unit_{c}\) consists
  in reassociating the binary composite of \(c\) and \(\inv c\). In the target,
  all the maximal cells that are composed in codimension \(1\) are associated
  together.
  First, we define a filler cell
  \[
    m^{1}_{0} = \filler(B_{i},f_{i},A_i, \sigma_{i}),
  \]
  using the following data. The first tree is
  \(B_{1} = \graft{}{D_{d}}{D_{d}}\) and the morphism
  \(f_{1}\colon B_1\to \bat\) picks \(B\) and \(B\). The full sphere \(A_1\)
  is the one defining the unbiased composite over \(B_{1}\), and
  \(\sigma_{1}\) is the map picking \(c_{0}\) and \(c'_{0}\).
  The second tree is \(B_{2} = \red{(\graft{}{B}{B})} = \red{B}\), and the
  second morphism \(f_{2}\) is the morphism that sends any maximal position \(p\)
  onto \(\susp^{d-1}\Chain_{\length_{\graft{}{B}{B}}(p)}\) and lower
  dimensional positions to disks of the same dimension. The sphere
  \(A_2 = a_{2}\to b_{2}\) is given by the cells
  \begin{align*}
    a_{2} &= T(s_{d-1}^{B_2})(a) &
    b_{2} &= T(t_{d-1}^{B_2})(a)
  \end{align*}
  and the map \(\sigma_2\colon \Pos{B_2}\to \Pos{\graft{}{B}{B}}\) is precisely
  \(\reduced_{\graft{}{B}{B}}\). One can check that
  \(\mu^{\str}(f_{1}) = \mu^{\str}(f_{2}) = \graft{}{B}{B}\) by induction on the
  tree. Moreover, the following equality
  \[ \type(\sigma_1)(A_1) = \type(\sigma_2)(A_2),\]
  a consequence of Lemma~\ref{lemma:bdry-reduction},
  justifies that the filler \(m^{1}_{0}\) is well defined. We define then
  \[m^{1} = (m^{1}_{0})^{\wcat X}(\pomap{\sigma,\bar\sigma}).\]
  The source and target of those cells are given by
  \begin{align*}
    \src(m^{1}_{0})
    &= \compcell{}{c_{0}}{c'_{0}} \\
    \tgt(m^{1}_{0})
    &= \coh{B_2}
      {T(s^{B_2}_{d-1})(a)\to T(t^{B_2}_{d-1})(a)}
      {\reduced_{\graft{}{B}{B}}} \\
    \src(m^{1})
    &= \compcell{d-1}{c}{\inv{c}} \\
    \tgt(m^{1})
    &= \coh{B_2}
      {T(s^{B_2}_{d-1})(a) \to
      T(t^{B_2}_{d-1})(a)}
      {\reduced_{\graft{}{B}{B}}}^{\wcat X}(\pomap{\sigma,\bar\sigma}).
  \end{align*}

  \paragraph{Telescope cancellation.}
  Our next step consists in cancelling away the codimension \(1\) composites of
  maximal cells into identities, using the telescopes. For that, we will use the
  functorialisation operation, and Lemma~\ref{lemma:functorialisation-po}.
  We denote by \(M = \Pos[d]{\graft{}{B}{B}}\) the set of maximal positions of
  the tree \(B_2 = \red{(\graft{}{B}{B})}\). We define also the maps
  \begin{align*}
    \tau_{\pm} &: \Pos{B_2} \to X \\
    \tau_{-} &= \alpha \circ T(\pomap{\sigma,\bar\sigma})\circ
      \cell(\reduced_{\graft{}{B}{B}})
      \circ \eta_{\Pos{B_2}} \\
    \tau_{+} &= \alpha \circ
      T(\pomap{\sigma,\bar\sigma}\circ s^{B_2}_{d-1})
      \circ \cell(\underline{\id})
      \circ \eta_{\Pos{B_2}}.
  \end{align*}
  To use Lemma~\ref{lemma:functorialisation-po}, it suffices to define
  for every position \(p\) of \(B_2\), a
  cell \(\tau_{p}\) such that \(\src(\tau_{p}) = \tau_{-}(p)\) and
  \(\tgt(\tau_{p}) = \tau_{+}(p)\). This provides a morphism
  \[
    \tau \colon \Pos{\fun{M}{B_2}} \to X
  \]
  with \(\tau\circ s^{B_2}_{d-1} = \tau_{-}\) and \(\tau\circ
  t^{B_2}_{d-1} = \tau_{+}\). This allows us to define a cell
  \[
    m^{2} = (\fun{M}{(B_2,
      T(s^{B_2}_{d-1})(a) \to
      T(t^{B_2}_{d-1})(a)) })^{\wcat X}(\tau).
  \]
  whose source and target are given by
  \begin{align*}
    \src(m^{2})
      &= \alpha\circ T(\tau_-)(\coh{B_2}{T(s^{B_2}_{d-1})(a) \to
        T(t^{B_2}_{d-1})(a)}{\id})\\
      &= \coh{B_2}
        {T(s^{B_2}_{d-1})(a) \to
        T(t^{B_2}_{d-1})(a)}
        {\reduced_{\graft{}{B}{B}}}^{\wcat X}(\pomap{\sigma,\bar\sigma}) \\
    \tgt(m^{2})
      &= \alpha\circ T(\tau_+)(\coh{B_2}{T(s^{B_2}_{d-1})(a) \to
        T(t^{B_2}_{d-1})(a)}{\id})\\
      &= \coh{B_2}
        {T(s^{B_2}_{d-1})(a) \to
        T(t^{B_2}_{d-1})(a)}
        {\underline{\id}}^{\wcat X}(\pomap{\sigma,\bar\sigma}\circ s^{B_2}_{d-1})
  \end{align*}
  by the \(\alpha\circ T\alpha = \alpha\circ \mu\), naturality, and the unit law
  of the free \(\omega\)\-category monad.

  To finish the construction of \(\tau\), we fix a maximal position \(p\in M\),
  and we
  let \(k = \length_{B}(p)\). Then by Lemma~\ref{lemma:red-comp},
  \(\length_{\graft{}{B}{B}}(p) = 2k\). By the assumption that
  \(c\in W_{n+1}\), we have for every maximal position
  \({f_i^{\Chain}\in \susp^{d-1}\Chain_{k}}\), that the cell
  \[c_{p,i} = \sigma_V(\chain_p(f_i^{\Chain}))\]
  belongs in \(W_{\le n}\). By the induction on \(n\), we may assume that the
  cells \(\inv{c_{p,i}} \in X_d\) and \(\unit_{c_{p,i}}\in W_\infty\) have
  already been constructed. This lets us define the morphism
  \[\cancel_{p}:\susp^{d-1} \Tel_{k} \to \wcat{X},\]
  by the following assignment:
  \begin{align*}
    x_{i}^{\Tel} & \mapsto \sigma_{V}(\chain_{p}(x_{i}^{\Chain})) &
    f_{i}^{\Tel} &\mapsto c_{p,i} &
    g_{i}^{\Tel} &\mapsto \inv{c}_{p,i} &
    \alpha_{i}^{\Tel} &\mapsto \unit_{c_{p,i}},
  \end{align*}
  and the cell
  \[\tau_{p} = \cancel_p(\susp^{\cell,d-1}\tel_k)\in W_{\infty}.\]
  Using Lemma~\ref{lemma:red-comp}, we get the following commutative diagram of
  globular sets:
  \[
    \begin{tikzcd}[row sep = small, column sep = large]
      & \Pos{\susp^{d-1}\Chain_{k}}
        \ar[dd, "\susp^{d-1}\inc^{-}"']
        \ar[dr, "\chain_{p}"]\\
      & &   \Pos{B}\ar[dd,"\inc^{-}"] \\
      \Pos{\susp^{d-1}\Chain_{k}}\ar[r,"\susp^{d-1}\inc^{+}"]
      \ar[dr, "\chain_{p}"']
      & \Pos{\susp^{d-1}\Chain_{2k}} \ar[dr,"\chain_{p}"]\\
      &  \Pos{B}\ar[r,"\inc^{+}"'] & \Pos{\graft{}{B}{B}}
    \end{tikzcd}
  \]
  from which we deduce the equality of morphisms
  \[\pomap{\sigma,\bar\sigma}\circ\chain_{\graft{}{B}{B},p} =
    \pomap{\sigma\circ\chain_{B,p},\bar\sigma\circ\chain_{B,p}}.\] Moreover, one
  can show that the following diagram commutes, by showing that it commutes on
  generators and that both sides are morphisms of \(\omega\)\-categories.
  \[
    \begin{tikzcd}
      T\Pos{\susp^{d-1}\Chain_{2k}}
      \ar[rr,"\cell(\susp^{d-1}\loopmap_{k})"],\ar[d,equal]
      && \cell(\susp^{d-1}\Tel_{k}) \ar[dd,"\cancel_{p}"]\\
      T\Pos{\graft{}{\susp^{d-1}\Chain_{k}}{\susp^{d-1}\Chain_{k}}}
      \ar[d,"T\pomap{\chain_{B,p},\chain_{B,p}}"'] \\
      T\Pos{\graft{}{B}{B}}\ar[r,"T\pomap{\sigma,\bar\sigma}"']
      & TX \ar[r,"\alpha"]& X
    \end{tikzcd}
  \]
  Together with the definition of \(\chain\) from
  Lemma~\ref{lemma:chain-reduce-map}, this shows that the source of
  \(\tau_p\) is given by:
  \begin{align*}
    \src(\tau_{p})
    &= \src (\cancel_p (\susp^{\cell,d-1}\tel_k)) \\
    &= \cancel_{p}(\cell(\susp^{d-1}\loopmap_{k}))
      (\susp^{\cell,d-1}\unbcomp_{\Chain_{2k}}) \\
    &= \alpha\circ T(\pomap{\sigma \circ \chain_{B,p},
      \bar{\sigma} \circ \chain_{B,p}})(\unbcomp_{\susp^{d-1}\Chain_{2k}})\\
    &= \alpha\circ T(\pomap{\sigma,\bar\sigma}\circ \chain_{\graft{}{B}{B},p})
      (\reduced_{\graft{}{B}{B},V}(p))\\
    &= \tau_{-}(p).
  \end{align*}
  Finally, we note that the position
  \(\chain_{B,p}(x_{0}^{\Chain}) \in\Pos[d-1]{\graft{}{B}{B}}\) is the unique
  position parallel to \(\src p\) which belongs to the image of
  \(s^{\graft{}{B}{B}}_{d-1}\), so by definition of \(\underline{\id}\), we
  deduce
  \[T(s^{\graft{}{B}{B}}_{d-1})(\underline{\id}_{V}(p)) =
  \id_{d-1}(\chain_{\graft{}{B}{B},p}(x_{0}^{\Chain})).\] The target of
  \(\tau_{p}\) is thus given by:
  \begin{align*}
    \tgt(\tau_{p})
    &= \tgt (\cancel_p (\susp^{\cell,d-1}\tel_k)) \\
    &= \cancel_{p}(\susp^{\cell,d-1}(\id_0(\var x_0^{\Tel}))) \\
    &= \id_{d-1}(\cancel_{p,V}(x_0^{\Tel}))\\
    &= \id_{d-1}(\sigma_{V}(\chain_{B,p}(x_{0}^{\Chain})))\\
    &= \id_{d-1}(\pomap{\sigma,\bar\sigma}\circ\chain_{\graft{}{B}{B},p}
      (\var x_{0}^{\Chain}))\\
    &= \alpha\circ T\pomap{\sigma,\bar\sigma} (\id_{d-1} (\chain_{\graft{}{B}{B},p}
      (\var x_{0}^{\Chain})))\\
    &= \alpha \circ T\pomap{\sigma,\bar\sigma}\circ Ts^{\graft{}{B}{B}}_{d-1}
      (\underline{\id}_{V}(p))\\
    &= \tau_{+}(p).
  \end{align*}

  \paragraph{Unbiased unitor.} The final step of our construction consists in
  the application of an unbiased unitor, which allow us to cancel the composite
  of identities that we have produced, into a single identity. Explicitly, we
  consider the cells
  \begin{align*}
    m^{3}_{0} &= \unitor(B_{2}, a)
    & m^{3} = (m^{3}_{0})^{\wcat X}(\pomap{\sigma,\bar\sigma} \circ
      s^{\graft{}{B}{B}}_{d-1}).
  \end{align*}
  The source and target of \(m^{3}\) are given by
  \begin{align*}
    \src(m^{3})
    &= \coh{B_2}{T(s^{B_2}_{d-1})(a) \to
      T(t^{B_2}_{d-1})(a)}
      {\underline{\id}}^{\wcat X}(\pomap{\sigma,\bar\sigma}\circ s^{B_2}_{d-1}) \\
    \tgt(m^{3}) &= \id(a)^{\wcat X}
                  (\pomap{\sigma,\bar\sigma}\circ s^{\graft{}{B}{B}}_{d-1}).
  \end{align*}
  This terminates the construction of \(\unit_{c}\). We observe further that
  \begin{align*}
    \partial_{d-1}(\graft{d-1}{B}{B}) &= \partial_{d-1}B  &
    \pomap{\sigma,\bar{\sigma}}\circ s_{d-1}^{\graft{}{B}{B}}
    &= \sigma \circ s_{d-1}^{B}
  \end{align*}
  by the axioms of the strict \(\omega\)\-category \(F^{\str}X\), which allows
  us to rewrite the target of \(m^3\) as
  \[\tgt(m^3) = \id(a)^{\wcat X}(\sigma \circ s_{d-1}^B) = \id(\src c),\]
  which concludes the proof.
\end{proof}

\begin{corollary}
  The relation \(\sim\) is an equivalence relation
\end{corollary}
\begin{proof}
  We have proven that this relation is symmetric in
  Lemma~\ref{lem:equiv-sym-and-functors} and reflexive
  in Corollary~\ref{cor:equiv-refl}, so it suffices to prove it is transitive.
  Consider three \(d\)\-cells \(c_{1},c_{2},c_{3}\) such that
  \(c_{1} \sim c_{2}\) and
  \(c_{2} \sim c_{3}\). Then there exist two invertible cells \(x,x'\) with
  \begin{align*}
    \src(x) &= c_{1} & \tgt(x) &= c_{2} &
    \src(x') &= c_{2} & \tgt(x') &= c_{3}.
  \end{align*}
  By Theorem~\ref{thm:inv-comp}, the cell \(\compcell{d}{x}{x'}\) is invertible
  and witnesses the relation \(c_1\sim c_3\).
\end{proof}

\subsection{Invertible cells of a computad}
We now consider the question of the invertibility of cells in an
\(\omega\)-category freely generated by a computad.
Proposition~\ref{prop:invertible-coh} and Theorem~\ref{thm:inv-comp}
applied in a computad give a syntactic recognition criterion for invertibility:
\begin{prop}\label{prop:invertible-computad}
  Consider a cell \(c \in \cell_d(C)\) in a computad \(C\). If
  \(\supp_d(c) = \emptyset\) then the cell \(c\) is invertible.
\end{prop}
\begin{proof}
  We proceed by structural induction on the cell \(c\). When \(c\) is a
  generator, the statement is vacuously true since the support can not be
  empty. Suppose therefore that \(c = \coh{B}{A}{\sigma}\) and that the
  result holds for the cell \(\sigma_V(p)\) for every \(p\in \Pos[d]{B}\). If
  no such position exists, then \(\dim B < d\) and \(c\) is invertible by
  Proposition~\ref{prop:invertible-coh}. Otherwise, the support of each
  \(\sigma_V(p)\) is empty, so it is invertible. By Theorem~\ref{thm:inv-comp},
  we conclude then that \(c\) is invertible as well.
\end{proof}

This sufficient condition may not be necessary in general, indeed in an
arbitrary computad, a generator may be invertible, violating this condition.
However, this condition is necessary in finite dimensional computads. To prove
this, we rely on the following technical result.
\begin{lemma}\label{lemma:support-noninv}
  Consider a cell \(c \in \cell_{d}(C)\) in a computad \(C\) and a cell \(c\).
  such that \(\supp_{d}(c) = \emptyset\), then
  \(\supp_{d-1}(\src (c)) = \supp_{d-1}(\tgt(c)) = \supp_{d-1}(c)\).
\end{lemma}
\begin{proof}
  We prove this result by structural induction on \(c\). Again, when \(c\) is a
  generator the statement is vacuous, because its \(d\)\-support is not empty.
  Suppose that \(c = \coh{B}{a\to b}{\sigma}\) and that for every \(p\in\Pos[d]{B}\)
  the result holds for the cell \(\sigma_{V}(p)\). If no such position exists, then
  \(\dim B < d\), so \(a\) covers \(\Pos{B}\). Thus, by the support
  lemma~\cite[Lemma~7.3]{dean_computads_2024}, we have that
  \begin{align*}
    \supp_{d-1}(c) &= \bigcup\limits_{p\in\Pos[d-1]{B}}
                     \supp_{d-1}(\sigma_{V}(p))
                   = \supp_{d-1}(\cell(\sigma)(a))
  \end{align*}
  Otherwise, \(\dim B = d\),  so there exists a cell
  \(a'\in T(\Pos{\bdry[d-1]B})_{d-1}\) covering \(\bdry[d-1]{B}\) such that
  \(
    a = T(s^{B}_{d})(a').
  \)
  Since \(a'\) covers \(\bdry[d-1]{B}\), we get by the support lemma that
  \begin{align*}
    \supp_{d-1}(c) &= \bigcup\limits_{p\in\Pos[d-1]{B}}
    \supp_{d-1}(\sigma_{V}(p)) \cup \bigcup\limits_{p\in\Pos[d]{B}}
    \supp_{d-1}(\sigma_{V}(p)) \\
    \supp_{d-1}(\src(c)) &= \bigcup\limits_{p\in\Pos[d-1]{\bdry[d-1]{B}}}
    \supp_{d-1}(\sigma_{V}(s_{d-1}^B(p))).
  \end{align*}
  The former clearly contains the latter, so it suffices to prove the converse
  inclusion. For every position \(p \in \Pos[d]{B}\), we have that
  \(\supp_d(\sigma_V(p)) \subset \supp_d(c) = \emptyset\), so by induction
  \[\supp_{d-1}(\sigma_{V}(p)) = \supp_{d-1}(\sigma_{V}(\src (p))).\]
  Hence, the second component in the union defining \(\supp_{d-1}(c)\) is
  superfluous. Moreover, by Lemma~\ref{lemma:id-everywhere},
  for every position \(p\in\Pos[d-1]{B}\), there exists a unique position
  \({q\in \Pos[d-1]{\bdry[d-1]{B}}}\) such that \(s^B_{d-1}(q)\) and \(p\) are
  parallel. By a similar inductive argument,
  one can show that there exists a sequence of consecutive \(d\)\-positions
  \(p_1,\dots,p_n\in \Pos[d]{B}\) with \(\src(p_1) = s_{d-1}^B(q)\) and
  \(\tgt(p_n) = p\). Therefore, by the inductive hypothesis,
  \[\supp_{d-1}(\sigma_V(s_{d-1}^B(q))) =
  \supp_{d-1}(\sigma_V(\tgt(p_1))) =
  \dots = \supp(\sigma_V(p)) \]
  This concludes the equality of the support of \(c\) and its source.
  A similar argument shows the result for the target.
\end{proof}

\begin{prop}
  In a finite dimensional computad \(C\), a cell \(c \in\cell_{d}(C)\) is
  invertible if and only if \(\supp_{d}(c) = \emptyset\).
\end{prop}
\begin{proof}
  Sufficiency of the support condition has already been established in
  Proposition~\ref{prop:invertible-computad}, so it remains to show necessity.
  For that, suppose that \(C\) has no generator above some dimension \(n\) and
  let \(c\in \cell_d(C)\) be invertible with \(\supp_d(c)\not=\emptyset\). Then
  by the assumption, \(d \le n\). Moreover, \(\supp_d(\src (\unit_c))\) is
  non-empty, since it contains \(\supp_d(c)\), while \(\supp_d(\tgt(\unit_c))\)
  is empty being the support of an identity. By Lemma~\ref{lemma:support-noninv},
  this implies that \(\supp_{d+1}(\unit_c)\) must be non-empty. Since \(\unit_c\)
  is invertible, iterating this argument, we get an invertible
  cell \(\unit^{n+1-d}_c\in \cell_{n+1}(C)\) with non-empty top-dimensional
  support. But this contradicts the hypothesis that \(C\) has no generators
  of dimension \(n+1\).
\end{proof}

\section{Implementation in CaTT}
\label{sec:implem}
The description of computads that we work with in this article can be
equivalently formulated as a dependent type theory called \catt, introduced by
Finster and Mimram~\cite{finster_typetheoretical_2017}. In fact the formulation
of computads proposed by Dean et al.~\cite{dean_computads_2024} was heavily
influenced by this dependent type theory, and the syntactic equivalence between
the two was proved by the two authors of the paper and
Sarti~\cite{benjamin_catt_2024}. An implementation of a typechecker for the
dependent type theory \catt is available and maintained by the first
author\footnote{\url{https://archive.softwareheritage.org/browse/origin/directory/?origin_url=https://github.com/thibautbenjamin/catt&timestamp=2024-06-16T21:31:06.869967\%2B00:00}}.
We have integrated the work presented in this article to the implementation, in
such a way that given a term in the theory whose variables are all of dimension
lower than the term, one can automatically compute its inverse or a witness of
equivalence for this term. In practice, the user has defined a term \verb|t|,
which happens to be invertible, they can access to the chosen inverse computed
by our algorithm by inputting \verb|I(t)|, and they can access the chosen
witness of equivalence by inputting \verb|U(t)|.

This allows for improved mechanisation of terms in \catt, complementing the
suspension and the functorialisation of terms that were already implemented. We
have assessed the relevance of this mechanisation principle on a practical
example: the definition of the term corresponding to the Eckmann-Hilton cell, as
well as its inverse and the witness that these two cancel each other. The choice
of this particular example is motivated by several considerations: They are
important examples in higher category theory, due to their connection with
topology and homotopy theory. These cells are complicated enough to define the
simplification will be significant on it, and is not definable in a pasting
scheme, making the mechanisation non-trivial. Yet they are among the simplest
examples with this property, and are reasonable to define even without this
mechanisation principle, making them a very good example.

\begin{figure}
  \centering
  \begin{tabular}{|c|cccc|}
    \hline
    name & LoC & LoC (ratio) & declarations & declarations (ratio) \\
    \hline
    vanilla & 531 & \(+0\%\) & 93 & \(+0\%\) \\
    \verb|s| & 470 & \(-11.5\%\) & 85 & \(-8.6\%\) \\
    \verb|sf| & 397 & \(-25.2\%\) & 76 &  \(-18.2\%\) \\
    \verb|sfb| & 378 & \(-28.8\%\) & 70 & \(-24.7\%\) \\
    \verb|sfbi| & 73 & \(-86.2\%\) & 16 & \(-82.8\% \) \\
    \verb|new| & 23 & \(-95.7\%\) & 10 & \(-89.2\% \) \\
    \hline
  \end{tabular}
  \caption{Mechanisation of the cancellation data \(\unit\) of the
    Eckmann-Hilton cell.}
  \label{fig:catt-mechanisation}
\end{figure}

We have formalised the Echmann-Hilton and its inverse as well as the \(\unit\)
cell for the equivalence, and verified them in \catt using various levels of
mechanisation. The results are compiled in Figure~\ref{fig:catt-mechanisation},
and the files that we used to assess these are available under the
\verb|examples/invertibility-paper/| directory of the repository. To assess
the complexity of a file, we use as proxies the number of lines of code written
in the file and the number of individual declaration the file has. The number of
declaration fails to account for the complexity of each of the declaration,
while the number of lines of codes magnifies this parameter by accounting for
line skips for code formatting and comments. Overall, these two proxies together
provide a reasonable proxy for the complexity of the definitions. The levels of
mechanisation that we consider are cumulative: ``vanilla'' has no automation,
and then \verb|s| indicates that the suspension is used, \verb|f| indicates that
the functorialisation is used, \verb|b| indicates that the compositions and
identities are taken as built-ins and \verb|i| indicates that inverses and
witnesses of composition are computed automatically using inverses. The ratios
are always considered against the vanilla case with no mechanisation at all.
Both of the chosen metrics indicate that the mechanisation of inverses is by far
the most efficient mechanisation principle that we have defined for this
example. It is also by far the most intricate of those mechanisation principles.
Since our example consists in defining a cell, its inverse, and the \(\unit\)
cell associated to the inversion, one would expect that adding the mechanisation
of the inverses to divide the size by \(3\). However, comparing \verb|sfbi| and
\verb|sfb|, we observe a diminution of \(80\%\) in the number of lines of code
and of \(77\%\) of the number of declarations. This can be explained by the fact
that the \(\unit\) cell is more complex than both the Eckmann-Hilton cell and
its inverse, and also by the fact that the inverse mechanisation also allows for
simplification in the definition of the Eckmann-Hilton cell itself. Using the
automation of the construction of opposites and inverses, we were also able to
propose a new construction of the Eckmann-Hilton cell, presented in the file
\verb|new|, which is even shorter.


\bibliographystyle{plainurl}
\bibliography{bibliography}

\end{document}